\documentclass[11pt,twoside]{article}

\usepackage{amsfonts,epsfig,graphicx}
\usepackage{amsmath,amssymb,amsthm}
\usepackage{fullpage}
\usepackage{epsf}
\usepackage{fancyheadings}
\usepackage{graphics}
\usepackage{amsfonts}
\usepackage{amsmath}
\usepackage{psfrag}
\usepackage{macros}
\usepackage{amsmath}
\usepackage{amsfonts}
\usepackage{psfrag}
\usepackage{multirow}
\usepackage[numbers]{natbib}
\usepackage{wasysym}
\usepackage{color}
\newcommand{\red}[1]{\textcolor{red}{#1}}

\usepackage[bookmarks=true,colorlinks,citecolor=blue,urlcolor=blue]{hyperref}

\newtheorem{theorem}{Theorem}
\newtheorem{lemma}{Lemma}
\newtheorem{proposition}{Proposition}

\begin{document}

\begin{center}
  {\LARGE{\bf{ Lower bounds on the performance of polynomial-time\\
algorithms for sparse linear regression}}}

  \vspace{1cm}

  {\large
\begin{tabular}{ccccc}
Yuchen Zhang$^\star$ & & Martin J.\ Wainwright$^{\star, \dagger}$ &&
Michael I. Jordan$^{\star,\dagger}$
\end{tabular}
}

  \vspace{.5cm}

  \texttt{\{yuczhang,wainwrig,jordan\}@berkeley.edu} \\

  \vspace{.5cm}

  {\large $^\star$Department of Electrical Engineering and
  Computer Science
  ~~~~ $^\dagger$Department of Statistics} \\
\vspace{.1cm}

  {\large University of California, Berkeley} \\

  \vspace{.5cm}

\today
\end{center}

\vspace*{.2cm}

\begin{abstract}
Under a standard assumption in complexity theory ($\np \not \subset
\ppoly$), we demonstrate a gap between the minimax prediction risk for
sparse linear regression that can be achieved by polynomial-time
algorithms, and that achieved by optimal algorithms.  In particular,
when the design matrix is ill-conditioned, the minimax prediction loss
achievable by polynomial-time algorithms can be substantially greater
than that of an optimal algorithm.  This result is the first known gap
between polynomial and optimal algorithms for sparse linear
regression, and does not depend on conjectures in average-case
complexity.
\end{abstract}


\section{Introduction}

The past decade has witnessed a flurry of results on the performance
of polynomial-time procedures, many of them based on convex
relaxation, that aim at solving challenging optimization problems that
arise in statistics.  The large majority of these results have been of
the positive variety, essentially guaranteeing that a certain
polynomial-time procedure produces an estimate with low statistical
error; see the overviews~\cite{ChaRec12,NegRavWaiYu12} for results of
this type.  Moreover, in many cases, the resulting bounds have been
shown to be minimax-optimal, meaning that no estimator can achieve
substantially smaller error.  More recently, however, this compelling
story has begun to develop some wrinkles, in that gaps have been
established between the performance of convex relaxations and the
performance of optimal methods, notably in the context of sparse PCA
and related sparse-low-rank matrix problems (e.g.,~\cite{AmiWai08,
  BerRig12, BerRig13,KraNadVil13,Oym12}).  The main contribution of
this paper is to add an additional twist to this ongoing story, in
particular by demonstrating a fundamental gap between the performance
of polynomial-time methods and optimal methods for high-dimensional
sparse linear regression.  Notably, in contrast with the recent work
of Rigollet and Berthet~\cite{BerRig13} on sparse PCA, and subsequent
results on matrix detection~\cite{ma2013computational}, both of which
are based on average-case complexity, our result is based only on a
standard conjecture in worst-case complexity theory.

Linear regression is a canonical problem in statistics: it is based on
observing a response vector $\yvec \in \real^\numobs$ and a design
matrix $\Xmat \in \real^{\numobs \times \usedim}$ that are linked via
the linear relationship
\begin{align}
\label{eqn:standard-linear-model}
\yvec & = \Xmat \thetastar + w.
\end{align}
Here the vector $w \in \real^\numobs$ is some form of observation
noise, and our goal is to estimate the unknown vector $\thetastar \in
\real^\usedim$, known as the regression vector.  Throughout this
paper, we focus on the standard Gaussian model, in which the entries
of the noise vector $w$ are i.i.d. $N(0, \sigma^2)$ variates, and the
case of deterministic design, in which the matrix $\Xmat$ is viewed as
non-random.  In the sparse variant of this model, the regression
vector is assumed to have a relatively small number of non-zero
coefficients.  In particular, for some positive integer $\kdim <
\usedim$, the vector $\thetastar$ is said to be $\kdim$-sparse if it
has at most $\kdim$ non-zero coefficients.  Thus, our model is
parameterized by the triple $(\numobs, \usedim, \kdim)$ of sample size
$\numobs$, ambient dimension $\usedim$, and sparsity $\kdim$.

Given a $\kdim$-sparse regression problem, the most direct approach
would be to seek a $\kdim$-sparse minimizer to the least-squares cost
$\|\yvec - \Xmat \theta\|_2^2$, thereby obtaining the $\ell_0$-based
estimator
\begin{align}
\label{EqnDefnEllZeroEstimator}
\thetazero & \defn \arg \min_{\theta \in \Ball_0(\kdim)} \|\yvec -
\Xmat \theta\|_2^2.
\end{align}
Note that this estimator involves minimization over the
$\ell_0$-``ball''
\begin{align}
\label{EqnDefnEllZeroBall}
\Ball_0(\kdim) & \defn \big \{ \theta \in \real^\usedim \, \mid \,
\sum_{j=1}^\usedim \Ind[\theta_j \neq 0] \leq \kdim \big \}
\end{align}
of $\kdim$-sparse vectors.  This estimator is not easy to compute in a
brute force manner, since there are ${\usedim \choose \kdim}$ subsets
of size $\kdim$ to consider.  More generally, it is known that
computing a sparse solution to a set of linear equations is an NP-hard
problem~\cite{natarajan1995sparse}, and this intractability result has
motivated the use of various heuristic algorithms and approximations.
Recent years have witnessed an especially intensive study of methods
based on $\ell_1$-relaxation, including the basis pursuit and Lasso
estimators~\cite{tibshirani1996regression, chen1998atomic}, as well as
the Dantzig selector~\cite{candes2007dantzig}.  Essentially, these
methods are based on replacing the
$\ell_0$-constraint~\eqref{EqnDefnEllZeroBall} with its
$\ell_1$-equivalent, in either a constrained or penalized form. All of
these estimators are based on relatively simple convex optimization
problems (linear or quadratic programs), and so can be computed in
polynomial time.  Moreover, in certain cases, the performance of
$\ell_1$-based methods have been shown to meet minimax-optimal lower
bounds~\cite{raskutti2011minimax}.

Despite this encouraging progress, there remain some intriguing gaps
in the performance of \mbox{$\ell_1$-based} procedures, perhaps most
notably when assessed in terms of their \emph{mean-squared prediction
  error} $\frac{1}{\numobs} \|\Xmat \thetahat - \Xmat
\thetastar\|_2^2$.  In order to bring this issue into sharper focus,
given an estimator $\thetahat$, suppose that we evaluate its
performance in terms of the quantity
\begin{align}
\label{EqnDefnMSE}
\MSE(\thetahat; \Xmat) & \defn \sup_{\thetastar \in \Ball_0(\kdim)}
\frac{1}{\numobs} \Exs \big[\|\Xmat \thetahat - \Xmat \thetastar\|_2^2
  \big],
\end{align}
where the design matrix $\Xmat$ remains fixed, and expectation is
taken over realizations of the noise vector $\wvec \sim N(0, \sigma^2
I_{\numobs \times \numobs})$.

The criterion~\eqref{EqnDefnMSE} assesses the performance of the
estimator $\thetahat$ uniformly over the set of all $\kdim$-sparse
regression vectors.  In terms of this uniform measure, the
$\ell_0$-based estimator~\eqref{EqnDefnEllZeroEstimator} is
known~\cite{BunWegTsyb07,raskutti2011minimax} to satisfy the bound
\begin{align}
\label{EqnEllZeroBound}
\MSE(\thetazero; \Xmat) & \precsim  \frac{\sigma^2 \, \kdim
  \log \usedim}{\numobs},
\end{align}
where $\precsim$ denotes an inequality up to a universal constant,
meaning independent of all problem dependent quantities.  A noteworthy
point is that the upper bound~\eqref{EqnEllZeroBound} holds for
\emph{any} fixed design matrix $\Xmat$.

By way of contrast, most $\ell_1$-based guarantees involve imposing
certain conditions on the design matrix $\Xmat$.  One of the most
widely used conditions is the restricted eigenvalue (RE)
condition~\cite{bickel2009simultaneous,van2009conditions}, which lower
bounds the quadratic form defined by $\Xmat$ over a subset of sparse
vectors (see equation~\eqref{EqnDefnRE} to follow for a precise
definition).  Under such an RE condition, it can be shown that the
Lasso-based estimator $\thetaone$ satisfies a bound of the form
\begin{align}
\label{EqnEllOneBound}
\MSE(\thetaone; \Xmat) & \precsim \frac{1}{\RECONSQ} \;
\frac{\sigma^2 \, \kdim \log \usedim}{\numobs},
\end{align}
where $\RECON(\Xmat) \leq 1$ denotes the restricted eigenvalue
constant~\eqref{EqnDefnRE}.  Comparison of this bound to the earlier
$\ell_0$-based guarantee~\eqref{EqnEllZeroBound} shows that the only
difference is the RE constant, which is a measure of the conditioning
of the matrix $\Xmat$.  However, from a fundamental point of view, the
conditioning of $\Xmat$ has no effect on whether or not a good sparse
predictor exists; for instance, a design matrix with two duplicated
columns is poorly conditioned, but the duplication would have no
effect on sparse prediction performance.

The difference between the bounds~\eqref{EqnEllZeroBound}
and~\eqref{EqnEllOneBound} leaves open various questions, both about
the performance of the Lasso (and other $\ell_1$-based methods), as
well as polynomial-time methods more generally.  Beginning with
$\ell_1$-based methods, one possibility is that existing analyses of
prediction error are overly conservative, but that the Lasso can
actually achieve the bound~\eqref{EqnEllZeroBound}, without the
additional RE term.  When the regression vector $\thetastar$ has a
bounded $\ell_1$-norm, then it is possible to achieve a prediction
error bound that does \emph{not} involve the RE
constant~\cite{BunWegTsyb07}, but the resulting rate is ``slow'',
decaying as $1/\sqrt{\numobs}$ instead of the rate $1/\numobs$ given
in equation~\eqref{EqnEllOneBound}.  Foygel and Srebro~\cite{FoySre11}
asked whether this slow rate could be improved without an RE
condition, and gave a partial negative answer in the case $\kdim = 2$,
constructing a $2$-sparse regression vector and a design matrix
violating the RE condition for which the Lasso prediction error is
lower bounded by $1/\sqrt{\numobs}$.  In this paper, we ask whether
the same type of gap persists if we allow for \emph{all
  polynomial-time estimators}, instead of just the Lasso.  Our main
result is to answer this question in the affirmative: we show that
there is a family of design matrices $\Xbad$ such that, under a
standard conjecture in computational complexity ($\np \not \subset
\ppoly$), for any estimator $\thetapoly$ that can be computed in
polynomial time, its mean-squared error is lower bounded as
\begin{align*}
\MSE(\thetapoly; \Xbad) & \succsim \frac{1}{\RECON^2(\Xbad)}
\frac{\sigma^2 \kdim^{1 - \HACKPAR} \log \usedim}{\numobs},
\end{align*}
where $\HACKPAR > 0$ is an arbitrarily small positive scalar.
Consequently, we see that there is a fundamental gap between the
performance of polynomial-time methods and that of the optimal
$\ell_0$-based method.

The remainder of this paper is organized as follows.  We begin in
Section~\ref{SecBackground} with background on sparse linear
regression and restricted eigenvalue conditions. We then introduce
some background on complexity theory, followed by the statement of our
main result in Section~\ref{SecMain}. The proof of the main theorem is
given in Section~\ref{sec:proof-of-main-theorem}, with more technical
results deferred to the appendices.


\section{Background and problem set-up}
\label{SecBackground}

We begin with background on sparse linear regression, then introduce restricted eigenvalue conditions.
These notions allow us to give a precise characterization
of the mean-squared prediction error that can
be achieved by the $\ell_0$-based algorithm and by a thresholded version of the Lasso algorithm.

\subsection{Estimators for sparse linear regression}

As previously described, an instance of the sparse linear regression
problem is based on observing a pair $(\yvec, \Xmat) \in \real^\numobs
\times \real^{\numobs \times \usedim}$ that are linked via the linear
equation~\eqref{eqn:standard-linear-model}, where the unknown
regression vector $\thetastar \in \real^\usedim$ is assumed to be
$\kdim$-sparse, and so belongs to the $\ell_0$-ball $\Ball_0(\kdim)$.
An estimator $\thetahat$ of the regression vector is a (measurable)
function $(\yvec, \Xmat) \mapsto \thetahat \in \real^\usedim$, and our
goal is to determine an estimator that is both $\kdim$-sparse, and has
low prediction error $\frac{1}{\numobs} \|\Xmat \thetahat - \Xmat
\thetastar\|_2^2$.  Accordingly, we let $\ALG(\kdim)$ denote the
family of all estimators that return vectors in $\Ball_0(\kdim)$.
Note that the $\ell_0$-based estimator $\thetazero$, as previously
defined in equation~\eqref{EqnDefnEllZeroEstimator}, belongs to the
family $\ALG(\kdim)$ of estimators.  The following result provides a
guarantee for this estimator: \\

\begin{proposition}[Prediction error for $\thetazero$]
\label{PropThetaZero}
There are universal constants \mbox{$\UNICON_j, j = 1, 2$} such for
any design matrix $\Xmat$, the $\ell_0$-based estimator $\thetazero$
satisfies
\begin{align}
\label{EqnThetaZeroBound}
\frac{1}{\numobs} \| \Xmat \thetazero - \Xmat \thetastar\|_2^2 & \leq
\UNICON_1 \frac{\sigma^2 \kdim \log \usedim}{\numobs} \qquad \mbox{for
  any $\thetastar \in \Ball_0(\kdim)$}
\end{align}
with probability at least $1 - 2 e^{-\UNICON_2 \kdim \log \usedim}$.
\end{proposition}

We also consider another member of the family $\ALG(\kdim)$---namely,
a thresholded version of the Lasso
estimator~\cite{tibshirani1996regression, chen1998atomic}.  The
ordinary Lasso estimate $\thetaregparn$ based on regularization
parameter $\regparn > 0$ is given by
\begin{align*}
\thetaregparn & \defn \arg \min_{ \theta \in \real^\usedim} \Big \{
\frac{1}{2 \numobs} \|\yvec - \Xmat \theta\|_2^2 + \regparn
\|\theta\|_1 \Big \}.
\end{align*}
In general, this estimator need not be $\kdim$-sparse, but a
thresholded version of it can be shown to have similar guarantees.
Overall, we define the \emph{thresholded Lasso} estimator $\thetathr$
based on the following two steps:
\begin{enumerate}
\item[(a)] Compute the ordinary Lasso estimate $\thetaregparn$ with
  $\regparn = 4 \sigma \sqrt{\frac{\log \usedim}{\numobs}}$.
\item[(b)] Truncate $\thetaregparn$ to its $\kdim$ entries that are
  the largest in absolute value, thereby obtaining the estimate
  $\thetathr$.
\end{enumerate}
By construction, the estimator $\thetathr$ belongs to the family
$\ALG(\kdim)$.  The choice of regularization parameter given in step
(a) is a standard one for the Lasso.




\subsection{Restricted eigenvalues and $\ell_1$-guarantees}

We now define the notion of a (sparse) restricted eigenvalue (RE), and
then discuss guarantees on the Lasso-based estimator $\thetathr$ that
hold under such an RE condition.  Restricted eigenvalues are defined
in terms of subsets $\PlainSset$ of the index set $\{1, 2, \ldots,
\usedim \}$, and a cone associated with any such subset.  In
particular, letting $\PlainSbar$ denote the complement of
$\PlainSset$, we define the cone
\begin{align*}
\ConeSet(\PlainSset) & \defn \big \{ \theta \in \real^\usedim \, \mid
\, \|\theta_{\PlainSbar}\|_1 \leq 3 \|\theta_{\PlainSset}\|_1 \big \}.
\end{align*}
Here $\|\theta_{\PlainSbar}\|_1 \defn \sum_{j \in \PlainSbar}
|\theta_j|$ corresponds to the $\ell_1$-norm of the coefficients
indexed by $\PlainSbar$, with $\|\theta_{\PlainSset}\|_1$ defined
similarly.  Note that any vector $\thetastar$ supported on
$\PlainSset$ belongs to the cone $\ConeSet(\PlainSset)$; in addition,
it includes vectors whose $\ell_1$-norm on the ``bad'' set
$\PlainSbar$ is small relative to their $\ell_1$-norm on $\PlainSset$.

\begin{mydefn}[Restricted eigenvalue (RE) condition]
Given triplet $(\numobs, \usedim, \kdim)$, the
matrix $\Xmat \in \real^{\numobs \times \usedim}$ is said to satisfy a uniform $\RECON$-RE condition if
\begin{align}
\label{EqnDefnRE}
\frac{1}{\numobs} \|X \theta\|_2^2 & \geq \RECON \|\theta\|_2^2 \qquad
\mbox{for all $\theta \in \bigcup
\limits_{|\PlainSset| = \kdim} \ConeSet(\PlainSset)$.}
\end{align}
Moreover, the restricted eigenvalue constant of $\Xmat$, denoted by $\RECON(\Xmat)$, is the greatest $\RECON$ such that
$X$ satisfies the condition~\eqref{EqnDefnRE}.
\end{mydefn}

The RE condition~\eqref{EqnDefnRE} and related quantities have been
studied extensively in past work on basis pursuit and the Lasso
(e.g.,~\cite{bickel2009simultaneous,meinshausen2009lasso,
  raskutti2011minimax}); see the paper~\cite{van2009conditions} for an
overview of the different types of RE parameters.  Note that it
characterizes the curvature of the quadratic form specified by $\Xmat$
when restricted to a certain subset of relatively sparse vectors.
When the RE constant $\RECON(\Xmat)$ is close to zero, there are relatively sparse vectors
$\thetatil$ such that $\|\Xmat \thetatil - \Xmat \thetastar\|_2$ is
small but $\|\thetatil - \thetastar\|_2$ is large.  Given that we
observe only a noisy version of the product $\Xmat \thetastar$, it is
then difficult to distinguish $\thetastar$ from other sparse vectors,
which makes estimating $\thetastar$ from the data difficult.  Thus, it is
natural to impose an RE condition if the goal is to produce an
estimate $\thetahat$ such that the $\ell_2$-norm error $\|\thetahat -
\thetastar\|_2$ is small.  Indeed, for $\ell_2$-norm estimation,
Raskutti et al.~\cite{raskutti2011minimax} show that a closely related
condition is necessary for any method.

In contrast, it is worth noting that the RE condition is not a
necessary condition for minimizing the prediction loss
$\norms{\Xmat\thetahat - \Xmat\thetastar}_2$, since an estimator far
apart from $\thetastar$ may still achieve small prediction
error. However, the RE condition turns to be an important criterion
for $\ell_1$-based methods to guarantee good prediction
performance. Under the normalization condition
\begin{align}
\label{EqnSuperAnnoy}
\frac{\|\Xmat \theta\|_2^2}{\numobs} & \leq \|\theta\|_2^2 \qquad
\mbox{for all $\theta \in \Ball_0(2 \kdim)$},
\end{align}
the following result provides such a guarantee for the thresholded
Lasso estimator:

\begin{proposition}[Prediction error for thresholded Lasso]
\label{PropLassoThresh}
There are universal constants \mbox{$\UNICON_j, j = 3, 4$} such that
for any design matrix $\Xmat$ satisfying the normalization
condition~\eqref{EqnSuperAnnoy} and having the RE constant
$\RECON(\Xmat) > 0$, the thresholded Lasso estimator $\thetathr$
satisfies
\begin{align}
\label{EqnLassoThreshBound}
\frac{1}{\numobs} \| \Xmat \thetathr - \Xmat \thetastar\|_2^2 & \leq
\frac{\UNICON_3}{\RECONSQ} \frac{\sigma^2 \kdim \log \usedim}{\numobs}
\qquad \mbox{for any $\thetastar \in \Ball_0(\kdim)$}
\end{align}
with probability at least $1 - 2 e^{-\UNICON_4 \kdim \log \usedim}$.
\end{proposition}
\noindent
See Appendix~\ref{AppLassoThresh} for the proof of this claim.  Apart
from different numerical constants, the main difference between the
guarantee~\eqref{EqnLassoThreshBound} and our earlier
bound~\eqref{EqnThetaZeroBound} for the $\ell_0$-based estimator is
the $1/\RECON^2(\Xmat)$ term.  The RE constant $\RECON(\Xmat)$ is a
dimension-dependent quantity, since it is a function of the $\numobs
\times \usedim$ design matrix.


\section{Main result and its consequences}
\label{SecMain}

Thus far, we have considered two estimators in the class
$\ALG(\kdim)$---namely, the $\ell_0$-constrained estimator
$\thetazero$ and the thresholded Lasso estimator $\thetathr$.  We also
proved associated guarantees on their prediction error
(Propositions~\ref{PropThetaZero} and~\ref{PropLassoThresh}
respectively), showing a $1/\RECONSQ$ gap between their respective
guarantees.  Our main result shows that this gap is \emph{not a
  coincidence}: more fundamentally, it is a characterization of the
gap between an optimal algorithm and the class of all polynomial-time
algorithms.\\

In order to state our main result, we need to make precise a
particular notion of a polynomial-efficient estimator. Since the
observation $(\Xmat,y)$ consists of real numbers, any efficient
algorithm can only take a finite-length representation of the input.
Consequently, we need to introduce an appropriate notion of
discretization, as has been done in past work on matrix
detection~\citep{ma2013computational}.  We begin by defining an
appropriate form of \emph{input quantization}: for any input value $x$
and integer $\quantlevel$, the operator
\begin{align*}
\floor{x}_\quantlevel & \defn 2^{-\quantlevel} \floor{2^\quantlevel x}
\end{align*}
represents a $2^{-\quantlevel}$-precise quantization of $x$. (Here
$\floor{u}$ denotes the largest integer smaller than or equal to $u$.)
Given a real value $x$, an efficient estimator is allowed to take
$\floor{x}_\quantlevel$ as its input for some finite choice
$\quantlevel$. We denote by $\size(x; \quantlevel)$ the length of
binary representation of $\floor{x}_\quantlevel$, and denote by
$\size(\Xmat,y; \quantlevel)$ the total length of the discretized
matrix vector pair $(X, y)$.

The following definition of efficiency is parameterized in terms of
three quantities: (i) a positive integer $\Lint$, corresponding to the
number of bits required to implement the estimator as a computer
program; (ii) a polynomial function $\pFunb$ of the triplet $(\numobs,
\usedim, \kdim)$, corresponding to the discretization accuracy of the
input, and (iii) a polynomial function $\pFunc$ of input size,
corresponding to the runtime of the program.

\begin{mydefn}[Polynomial-efficient estimators]
\label{DefnEfficientEstimators}
Given polynomial functions $\pFunb: (\PosInt)^3 \rightarrow \real_+$,
$\pFunc: \PosInt \rightarrow \real_+$ and a positive integer $\Lint
\in \PosInt$, an estimator $(y, X) \mapsto \thetahat(y, X)$ is said to
be \emph{$(\Lint, \pFunb, \pFunc)$-efficient} if:
\begin{itemize}
\item It can be represented by a computer program that is encoded in
  $\Lint$ bits.
\item For every problem of scale $(\numobs, \usedim, \kdim)$, it
  accepts inputs quantized to accuracy $\floor{\cdot}_\quantlevel$
  where the quantization level is bounded as $\quantlevel \leq
  \pFunb(\numobs, \usedim, \kdim)$.
\item For every input $(\Xmat,y)$, it is guaranteed to terminate in
  time $\pFunc(\size(\Xmat,y;\quantlevel))$.
\end{itemize}
\end{mydefn}
According to this definition, if we choose a sufficiently large code
length---say $\Lint =10^{16}$---and polynomial functions $\pFunb$ and
$\pFunc$ that grow sufficiently fast---say \mbox{$\pFunb(\numobs,
  \usedim, \kdim) = (\numobs \usedim \kdim)^{100}$} and
\mbox{$\pFunc(s) = s^{100}$}---then the class of $(\Lint, \pFunb,
\pFunc)$-efficient algorithms allow estimators to take sufficiently
accurate input, and covers all estimators that are reasonably
efficient in terms of storage and running time.

To present the main result, we require a few more notions from
complexity theory.  See the book~\cite{arora2009computational} for a
more detailed introduction.  In complexity theory, the class $\poly$
corresponds to problems that are solvable in polynomial time by a
Turing machine.  A closely related class denoted by \ppoly,
corresponds to all problems solvable in polynomial time by a Turing
machine with a so-called advice string (meaning a side-input to the
machine) that is of polynomial length.  Although it is known that the
class \ppoly~is strictly bigger than the class \poly
~(e.g,~\cite{arora2009computational}), it is widely believed that $\np
\not \subset \ppoly$.  Accordingly, throughout this paper, we impose
the following assumption: \\

\begin{assumption}
\label{assumption:np-notin-ppoly}
The class $\np$ is not contained in the class $\ppoly$.
\end{assumption}

\noindent Based on Assumption~\ref{assumption:np-notin-ppoly}, we are
ready to present the main result.  In this statement, we use
$\UNICON_j, j = 5, 6$ to denote universal constants independent of
$(\numobs, \usedim, \kdim)$, $(\pFuna,\pFunb, \pFunc)$ and $(\RECON, \sigma,\hackpar)$.

\begin{theorem}
\label{ThmMain}
If $\np \not \subset \ppoly$, then for any positive integer $\Lint$,
any scalar $\hackpar \in (0,1)$, any polynomial functions
\mbox{$\pFunb: (\PosInt)^3 \rightarrow \real_+$} and
\mbox{$\pFuna,\pFunc: \PosInt \rightarrow \real_+$}, there is a
sparsity level $\kdim \geq 1$ such that the following holds: \\

\noindent For any dimension \mbox{$\usedim \in [4\kdim,
    \pFuna(\kdim)]$,} any sample size $\numobs$ in the interval
\mbox{$[\UNICON_5 \kdim \log \usedim, \pFuna(\kdim)]$}, and any
scalar \mbox{$\RECON \in [2^{-\pFunb(\numobs,\usedim,\kdim)},
    \frac{1}{24 \sqrt{2}})$,} there is a matrix $\Xmat \in \R^{\numobs
    \times \usedim}$ such that
\begin{enumerate}
  \item[(a)] It satisfies the normalization
    condition~\eqref{EqnSuperAnnoy}, and has an RE constant
    $\RECON(\Xmat)$ that is bounded as \mbox{$| \RECON(\Xmat) - \RECON |
      \leq 2^{-\pFunb(\numobs,\usedim,\kdim)}$.}
  \item[(b)] For any $(\Lint, \pFunb, \pFunc)$-efficient estimator
    $\thetahat\in \ALG(\kdim)$, the mean-squared prediction error is
    lower bounded as
\begin{align}
\label{eqn:main-theorem-lower-bound}
\max_{\thetastar \in \Ball_0(\kdim)} \Exs \Big [\frac{\|\Xmat
    (\thetahat - \thetastar)\|_2^2}{\numobs} \Big] & \geq
\frac{\UNICON_6}{\RECON^2} \; \frac{\sigma^2 \kdim^{1-\hackpar} \log
  \usedim}{\numobs}.
  \end{align}
\end{enumerate}
\end{theorem}

To understand the consequence of Theorem~\ref{ThmMain}, suppose that
we have chosen $\Lint$, $\pFuna$, $\pFunb$, $\pFunc$ sufficiently
large and have chosen $\hackpar$ very close to zero.  Under this
setting, Theorem~\ref{ThmMain}(b) shows that as long as the triplet
$(\numobs, \usedim, \kdim)$ is sufficiently large, then there is a
explicitly constructed design matrix $\Xmat$ such that any $(\Lint,
\pFunb,\pFunc)$-polynomial-efficient estimator has prediction risk
lower bounded by inequality~\eqref{eqn:main-theorem-lower-bound}.
Part (a) guarantees that the constructed design matrix $\Xmat$
satisfies the normalization~\eqref{EqnSuperAnnoy} and RE
conditions~\eqref{EqnDefnRE}, so that
Proposition~\ref{PropLassoThresh} can be applied to the thresholded
Lasso estimator for this instance of the sparse regression problem.
Since the parameter $\hackpar \in (0,1)$ may be chosen arbitrarily
close to zero, the lower bound~\eqref{eqn:main-theorem-lower-bound}
essentially matches the upper bound of
Proposition~\ref{PropLassoThresh}, thereby confirming that
Theorem~\ref{ThmMain} gives a tight lower bound.

Overall, Theorem~\ref{ThmMain} establishes that the inverse dependence
on the RE constant $\RECON(\Xmat)$ is unavoidable for the class of
polynomial-time algorithms.  In contrast, the $\ell_0$-based
estimator---a method that is not polynomially efficient---does
\emph{not} exhibit this dependence, as shown by
Proposition~\ref{PropThetaZero}. \\

It is worth contrasting Theorem~\ref{ThmMain} with past work on the
computational hardness of sparse linear systems.  As mentioned
previously, Natarajan~\cite{natarajan1995sparse} showed that finding
sparse solutions to linear systems is an NP-hard problem.  To contrast
with our result, this earlier result applies to the problem of solving
$\Xmat \theta = y$, where $\Xmat$ is the worst-case matrix selected by
an adversary who knows the algorithm in advance.  In contrast, in our
Theorem~\ref{ThmMain}, the design matrix $\Xmat$ is fixed ahead of
time, and the associated hardness result applies to all
polynomial-time algorithms.  Furthermore, Theorem~\ref{ThmMain}
requires $\Xmat$ to satisfy the normalization
condition~\eqref{EqnSuperAnnoy} as well as the RE
condition~\eqref{EqnDefnRE}, so that the lower bound is achievable by
the Lasso-based estimator in Proposition~\ref{PropLassoThresh}. These
two conditions are not satisfied in the earlier
construction~\cite{natarajan1995sparse}. Finally, in
Theorem~\ref{ThmMain} we have proved a lower bound of the order
$\frac{\sigma^2 \kdim\log \usedim}{\RECONSQ \, \numobs}$, which is
stronger than the $\frac{1}{\numobs}$ lower bound proved in the
original paper~\cite{natarajan1995sparse}.


\section{Proof of Theorem~\ref{ThmMain}}\label{sec:proof-of-main-theorem}

We now turn to the proof of our main result.  In broad outline, the
proof involves three main steps.  We begin in
Section~\ref{sec:construct-np-hard-problem} by constructing a
particular matrix $M$ for which the sparse linear regression problem
is NP-hard, doing so by comparison to the exact 3-set cover problem.
In Section~\ref{SecProbCase}, we then extend this worst-case hardness
result to the probabilistic setting that is of interest in
Theorem~\ref{ThmMain}.  Section~\ref{sec:define-distribution} is
devoted to the actual construction of the design matrix $\Xmat$, using
the matrix $M$ as a building block, and verifies that it satisfies the
conditions in part (a). Section~\ref{SecProofB} contains the proof of
part (b) of the theorem.  In all cases, we defer the proofs of more
technical lemmas to the appendices.


\subsection{An NP-Hard Problem}
\label{sec:construct-np-hard-problem}

In this section, we construct a linear system and prove that solving
it under a specific sparsity condition is NP-hard.  The NP-hardness is
establishing by reducing the exact 3-set cover
problem~\cite{natarajan1995sparse} to this linear regression problem.
The exact 3-set cover (X3C) problem is stated as follows: given an
integer $\mdim \geq 3$ divisible by three, a set $\SET =\{1,\dots,
\mdim\}$ and a collection $\COVER$ of 3-element subsets of $\SET$,
find $\mdim/3$ subsets in $\COVER$ that exactly cover $\SET$, assuming
such an exact cover exists.

Throughout our development, we make use of the convenient shorthand
$\pdim \defn {\mdim \choose 3}$, as well as $[\mdim] \defn \{1, 2,
\ldots, \mdim\}$ and $[\pdim] \defn \{1, 2, \ldots, \pdim\}$.  We
begin by constructing a matrix $M \in \real^{(\mdim + 3 \pdim) \times
  4 \pdim}$ in a blockwise manner, namely
\begin{align*}
M & \defn \begin{bmatrix} A & 0 \\ B & C \\
\end{bmatrix}
\end{align*}
where the submatrices have dimensions $A \in \R^{\mdim \times \pdim}$,
$B \in \R^{3\pdim \times \pdim}$ and $C \in R^{3 \pdim \times 3
  \pdim}$.

We define the submatrices by considering all possible subsets of the
form $(a,b,c) \in [\mdim]^3$, where the elements are all required to
be distinct.  Since $\pdim = {\mdim \choose 3}$, each such subset can
be labeled with a unique index $j \equiv j_{abc} \in [\pdim]$.  For
each $j \in [\pdim]$, this indexing can be used to define the $j^{th}$
column of the submatrix $A$ as follows:
\begin{align*}
  A_{aj} = A_{bj} = A_{cj} = 1 ~~~~\mbox{and}~~~~ A_{dj} =
  0~~\mbox{for all $d \in [\mdim] \backslash \{a,b,c\}$.}
\end{align*}
In other words, the $j^{th}$ column of $A$ is the binary indicator
vector for membership in the subset $\{a,b,c\}$ indexed by $j$.

The submatrices $B$ and $C$ are defined in terms of their rows.  For
each $j \in [\pdim]$, we define the following three rows of $B$
\begin{align*}
  B_{j} = e_j,~~B_{p+j}=e_j~~\mbox{and}~~B_{2 \pdim + j} = 0,
\end{align*}
where $e_j \in \real^{\pdim}$ is $j^{th}$ canonical basis vector
(i.e., with $j$-th entry equal to $1$ and all other entries set to
$0$).  Similarly, we define the rows of $C$ by
\begin{align*}
C_{j} = -f_j,~~C_{\pdim + j} = f_{p + j}~~\mbox{and}~~C_{2 \pdim + j}
= f_{2\pdim + j},
\end{align*}
where $f_j \in \real^{3 \pdim}$ is the $j^{th}$ canonical basis vector
in $\real^{3 \pdim}$.  We also define the set
\begin{align*}
\RESP & \defn \big \{ \resp \in \real^{\mdim + 3 \pdim} \, \mid \, \resp = M
\uvec~\mbox{for some $ \uvec \in \{0,1\}^{4p}\cap \Ball_0(m/3+p)$}
\big\}.
\end{align*}
The following lemma shows that solving the linear system $M \uvec =
\resp$ for all $\resp \in \RESP$ is NP-hard.

\begin{lemma}
\label{lemma:np-hardness-of-solving-linear-equation}
Given the matrix $M \in \real^{(\mdim + 3 \pdim) \times 4 \pdim}$ as
previously defined and a vector $\resp \in \RESP$, the problem of
computing a vector $\uvec \in \{0,1\}^{4 \pdim} \cap
\Ball_0(\frac{\mdim}{3} +\pdim)$ such that $\norm{M \uvec - \resp}_2 <
\frac{1}{2}$ is NP-hard.
\end{lemma}

\begin{proof}
We reduce the X3C problem, specified by the set $\SET = \{1, \ldots,
\mdim\}$ and triplet collection $\COVER$ to the sparse linear
regression problem.  In particular, given $(\SET, \COVER)$, we use it
to construct a response vector $\resp \in \real^{\mdim + 3 \pdim}$, and then
consider the linear system $M \uvec = \resp$.  We then show that any
method that leads to a sparse vector $\uvec$ such that $\|M \uvec -
\resp\|_2 < \frac{1}{2}$ can also be used to solve the exact 3-set
cover problem.

\paragraph{Constructing a response vector from X3C:}

Given $(\SET, \COVER)$, we now construct the response \mbox{vector
  $\resp$.}  Recalling that $\resp \in \real^{\mdim + 3 \pdim}$, we let the
first $\mdim$ coordinates of the vector $\resp$ equal to $1$. Since
any triplet of distinct elements $(a,b,c) \in [\mdim]^3$ can be
associated with a unique index $j_{abc} \in [\pdim]$.  We use this
correspondence to define the remaining $3 \pdim$ entries of $\resp$ as
follows:
\begin{subequations}
\begin{align}
\label{eqn:response-def-1}
\mbox{If $(a,b,c) \in \COVER$:} & \qquad \mbox{set $\resp_{\mdim +
    j_{abc}} = 0$, $\resp_{\mdim + \pdim + j_{abc}}=1$ and $\resp_{\mdim + 2 \pdim +
    j_{abc}}=0$,} \\
\label{eqn:response-def-2}
\mbox{If $(a,b,c) \notin \COVER$:} & \qquad \mbox{set $\resp_{\mdim +
    j_{abc}} = 0$, $\resp_{\mdim+ \pdim + j_{abc}} = 0$, and $\resp_{\mdim + 2
    \pdim + j_{abc}} = 1$.}
\end{align}
\end{subequations}

Assuming the existence of an exact cover, we now show that $\resp \in
\RESP$, in particular by constructing a vector $u \in \ANNOY$ such
that $M u = \resp$. Given a fixed triplet $(a,b,c)$, let us introduce
the shorthand notation
\begin{align}
\label{eqn:subset-variable-shorthands}
 \uone \defn \uvec_{j_{abc}},~~ \vone \defn \uvec_{p+i_{abc}},~~\wone
 \defn \uvec_{2 \pdim +j_{abc}}~~\mbox{and}~~ \zone \defn \uvec_{3
   \pdim + j_{abc}}.
\end{align}
Observe that the linear equation $M \uvec= \resp$ holds if and only if the following
conditions hold:
\begin{subequations}
\begin{align}
\mbox{For $(a,b,c) \in \COVER$:}& \qquad \uone - \vone = 0,~~ \uone +
\wone = 1, \; \mbox{and $\zone = 0$.} \label{eqn:solution-condition-1} \\
\mbox{For $(a,b,c) \notin \COVER$:}& \qquad \uone- \vone = 0,~~ \uone
+ \wone = 0, \; \mbox{and $\zone = 1$.} \label{eqn:solution-condition-2}\\
\mbox{For any $i\in \SET$:}& \qquad \mbox{exactly one triplet $(a,b,c)$ satisfies } i\in\{a,b,c\} \mbox{ and } \uone = 1.\label{eqn:solution-condition-3}
\end{align}
\end{subequations}
Thus, it is sufficient to construct a vector $\uvec \in \ANNOY$ which
satisfies the above conditions.  If $(a,b,c)\notin \COVER$, then we let
$(\uone,\vone,\wone,\zone)=(0,0,0,1)$. If $(a,b,c)\in \COVER$ but it does
not belong to the exact cover, then we let
$(\uone,\vone,\wone,\zone)=(0,0,1,0)$. Otherwise, if $(a,b,c)\in \COVER$
and it is selected in the exact cover, then we let
$(\uone,\vone,\wone,\zone) = (1,1,0,0)$.  Given these specifications,
it is straightforward to verify that $\uvec \in \{0,1\}^{4p}$ and
exactly $m/3+p$ entries of $\uvec$ are equal to $1$. It is also easy
to verify that $\uvec$ satisfies
conditions~\eqref{eqn:solution-condition-1}-\eqref{eqn:solution-condition-3}.  Overall, we conclude that
$\resp \in \RESP$, and hence is a valid response vector for the sparse linear regression
problem

\paragraph{Solving X3C using a sparse linear system solver:}
Suppose that there is a sparse solution $\uvec \in \ANNOY$ which
satisfies $\norm{M \uvec-\resp}_2 < 1/2$, where the response vector
$\resp$ was previously defined in equations~\eqref{eqn:response-def-1}
and~\eqref{eqn:response-def-2}. We now use $\uvec$ to construct an
exact 3-set cover.  Recalling the
notation~\eqref{eqn:subset-variable-shorthands}, we observe that the
condition $\norm{M \uvec-\resp}_2 < 1/2$ leads to the following
restrictions:
\begin{align*}
\mbox{For $(a,b,c) \in \COVER$:} & \qquad |\uone - \vone| < 1/2,~~ |\uone +
\wone - 1| < 1/2, \; \mbox{and} \; |\zone| < 1/2.\\
\mbox{For $(a,b,c) \notin \COVER$:} & \qquad |\uone- \vone| < 1/2,~~
|\uone + \wone| < 1/2, \; \mbox{and} \; |\zone - 1| < 1/2.
\end{align*}
Then, we prove the following two claims:
\begin{claim}
\label{eqn:exact-coverage-claim}
 There are exactly $\mdim/3$ nonzero entries in $\{\uvec_1, \dots,
 \uvec_\pdim\}$. In addition, for any $i \in [\pdim]$ such that
 $\uvec_i > 0$, we must have $\uvec_i > 1/2$.
\end{claim}

\begin{proof}
For any triplet $(a,b,c)$, we claim that there is at least one
nonzero entry in $(\vone,\wone,\zone)$. Thus, the fact that $\uvec\in
\Ball_0(m/3+p)$ implies that there are at most $m/3$ nonzero entries
among $\{\uvec_1,\dots,\uvec_p\}$. To prove this claim, consider two
cases $(a,b,c)\notin \COVER$ and $(a,b,c)\in \COVER$. If $(a,b,c)\notin \COVER$, then
we have a restriction $|\zone-1| < 1/2$, which makes $\zone\neq 0$. If
$(a,b,c)\in \COVER$, then we have restrictions $|\uone-\vone| < 1/2$ and
$|\uone+\wone - 1| < 1/2$.  Thus,
\begin{align*}
  1 > |\uone-\vone| + |\uone+\wone-1| \ge |\vone+\wone-1|,
\end{align*}
which implies that either $\vone$ or $\wone$ must be nonzero.

On the other hand, we claim that the number of nonzero entries in
$\{\uvec_1,\dots,\uvec_p\}$ is at least $m/3$. Assume that there are
$k$ nonzero entries in $\{\uvec_1,\dots,\uvec_p\}$.
We notice that each nonzero entry in $\{\uvec_1,\dots,\uvec_p\}$
selects a column in matrix $A$ that has three nonzero coordinates.
Thus, there are at most $3k$ nonzero entries in $\{(M\uvec)_1,\dots,(M\uvec)_m\}$.
Since $|(M\uvec)_i - 1| < 1/2$ holds for all $1\leq i\leq m$, all elements
in $\{(M\uvec)_1,\dots,(M\uvec)_m\}$ must be greater than $1/2$. Thus,
we have
$3k\geq m$ which establishes the claim.

Hence, there are exactly $m/3$ nonzero entries in $\{\uvec_1,\dots,\uvec_p\}$.
Each nonzero entry covers exactly three coordinates of $\resp$.
Since all elements
in $\{(M\uvec)_1,\dots,(M\uvec)_m\}$ must be greater than $1/2$,
we conclude that the nonzero $\uvec_i$'s must be greater than $1/2$.
\end{proof}

\begin{claim}
\label{eqn:subset-selection-consistency}
For any triplet $(a,b,c) \notin \COVER$, we have $\uvec_{i_{abc}}=0$.
\end{claim}

\begin{proof}
We proceed via proof by contradiction.  Suppose that $(a,b,c) \notin \COVER$ and $\uvec_{i_{abc}} \neq 0$. By
Claim~\ref{eqn:exact-coverage-claim}, we have \mbox{$\uone =
  \uvec_{i_{abc}} > 1/2$.}  The condition $\norm{M\uvec-\resp}_2 <
1/2$ imposes the restrictions \mbox{$|\uone-\vone| < 1/2$} and
\mbox{$|\zone-1| < 1/2$,} which in turn imply that $\vone \neq 0$ and
$\zone\neq 0$. In the proof of Claim~\ref{eqn:exact-coverage-claim},
we have shown that for any triplet $(a,b,c)$ there is at least one
nonzero entry in $(\vone,\wone,\zone)$. For this specific triplet
$(a,b,c)$, there are at least two nonzero terms among
$(\vone,\wone,\zone)$.  Hence, the total number of nonzero entries in
$\{\uvec_1,\dots,\uvec_p\}$ is at most $m/3-1$, which contradicts
Claim~\ref{eqn:exact-coverage-claim}.
\end{proof}

According to Claim~\ref{eqn:subset-selection-consistency}, all nonzero
entries in $\{\uvec_1,\dots,\uvec_p\}$ correspond to 3-sets in
$\COVER$. According to Claim~\ref{eqn:exact-coverage-claim}, exactly
$m/3$ of these 3-sets are selected. The condition
$\norm{M\uvec-\resp}_2 < 1/2$ implies that all entries in
$\{(M\uvec)_1,\dots,(M\uvec)_m\}$ are nonzero, which means that $\SET$
is covered by the union of these 3-sets.  Thus, these $m/3$ 3-sets form an
exact cover \mbox{of $\SET$.}
\end{proof}


\subsection{Probabilistic Hardness}
\label{SecProbCase}

As in the previous section, we use the shorthand notation $\pdim =
{\mdim \choose 3}$ throughout.  In this section, we extend the result
of Section~\ref{sec:construct-np-hard-problem} to the probabilistic
case.  In order to do so, we require the following auxiliary lemma:
\begin{lemma}
\label{lemma:ppoly-consequence}
Let $\pfuna, \pfunb$ be arbitrary polynomial functions, let $\mdim_0
\in \PosInt$ be an arbitrary positive integer, and suppose that
Assumption~\ref{assumption:np-notin-ppoly} holds.  Then for any
NP-hard problem $\PROB$, there is an input length $\mdim > \mdim_0$
such that any algorithm that can be encoded in $\pfuna(\mdim)$ bits
and that terminates in $\pfunb(\mdim)$ time must fail on at least one
input with probability greater than $1/2$.
\end{lemma}
\noindent See Appendix~\ref{AppPpolyConsequence} for the proof.\\

\noindent We are now ready to state and prove the main result of this
section:

\begin{lemma}
\label{lemma:probablistic-hardness}
Let $\pfuna, \pfunb$ and $\pfunc$ be arbitrary polynomial functions
defined on $\PosInt \defn \{1, 2, \ldots \}$.  Then for any $\mdimzero
\in \PosInt$, there is an integer $\mdim > \mdimzero$ and a distribution
$\qprob$ over $\uvec \in \{0,1\}^{4 \pdim} \cap
\Ball_0(\frac{\mdim}{3} + \pdim)$ such that for any solver $\uhat$
that can be encoded in $\pfuna(\mdim)$ bits and terminates in
$\pfunb(\mdim)$ time, we have
\begin{align}
\label{EqnHbound}
\qprob \Big[ \norm{M \uhat -M \uvec}_2 < \frac{1}{2} \Big] & \leq
\frac{1}{\pfunc(\mdim)}.
\end{align}
\end{lemma}

\begin{proof}
We proceed via proof by contradiction.  In particular, for a given
$\mdimzero$, suppose that for every $\mdim > \mdimzero$ and for any
distribution over $\{0,1\}^{4p}\cap \Ball_0(m/3+p)$, there is a solver
$\uhat$ that can be encoded in $\pfuna(\mdim)$ bits and terminates in
$\pfunb(\mdim)$ time, and such that inequality~\eqref{EqnHbound} is
violated.

For a positive integer $\MYINT$ to be specified, we now construct a
sequence of distributions $\{\mprob_{t}\}_{t=1}^\MYINT$ that lead to a
contradiction.  Let $\mprob_1$ the uniform distribution over $\Sset_1
\defn \{0,1\}^{4 \pdim }\cap \Ball_0(\frac{\mdim}{3} + \pdim)$ and let
$\uhat_1$ be the solver which achieves $\mprob_1 \big[\norm{M \uhat_1
    -M \uvec}_2 < \frac{1}{2} \big] > \frac{1}{\pfunc(\mdim)}$.  For
$t \geq 1$, we define the set $\Sset_{t+1}$ recursively as
\begin{align*}
\Sset_{t+1} & = \Sset_{t} \cap \big \{{\uvec} \in \Sset_0 \, \mid \,
\norm{M \uhat_t -M \uvec}_2 \geq \frac{1}{2} \big \}.
\end{align*}
When $\Sset_{t+1}$ is non-empty, we let $\mprob_{t+1}$ be the uniform
distribution over $\Sset_{t+1}$, and let $\uhat_{t+1}$ denote the
solver that satisfies the bound $\mprob_{t+1} \big [\norm{M \uhat-M
    \uvec}_2 < \frac{1}{2} \big] > \frac{1}{\pfunc(\mdim)}$.  If
$\Sset_{t+1}$ is empty, then we simply set $\uhat_{t+1} = \uhat_t$.

For any integer $\MYINT \geq 1$, this construction yields a sequence
of solvers $\{\uhat_t\}_{t=1}^\MYINT$, and we use them to define a
combined solver $\ustar_\MYINT$ as follows:
\begin{align}
\label{eqn:construct-cascade-search}
\ustar_\MYINT & = \begin{cases} \uhat_t & \mbox{for first $t \in \{1,
    \ldots, \MYINT \}$ such that $\norm{M \uhat_t - M \uvec} <
    \frac{1}{2}$ } \\ 0 & \mbox{otherwise.}
\end{cases}
\end{align}
If $\Sset_t$ is empty for some $t \in [\MYINT]$, then $\|M
\ustar_\MYINT - M u \|_2 < 1/2$ for all $u \in \Sset_0$.  Otherwise,
we may assume that $\Sset_t$ is non-empty for each $t \in [\MYINT]$.
Since $\Sset_t$ is the set of vectors $\uvec$ for which all the
solvers $\{\uhat_1, \ldots, \uhat_{t-1} \}$ fail, and $\mprob_t$ is
the uniform distribution over $\Sset_t$, the chance of success for
every solver $\uhat_t$ in the
construction~\eqref{eqn:construct-cascade-search} is at least
$\frac{1}{\pfunc(\mdim)}$. Consequently, if we consider the chance of
success for the solver $\ustar_\MYINT$ under the uniform distribution
over $\Sset_1$, we have
\begin{align}
\label{EqnBearCreekSpire}
\mprob_1 \big[\norm{M \ustar_\MYINT -M \uvec}_2 < \frac{1}{2} \big] &
\stackrel{(i)}{\geq} 1- \big(1- \frac{1}{\pfunc(\mdim)} \big)^\MYINT
\; \stackrel{(ii)}{\geq} 1-(1/2)^{4p+1},
\end{align}
where inequality (ii) follows by choosing $\MYINT \defn \lceil
\frac{-(4\pdim + 1)}{\log_2(1- \frac{1}{\pfunc(\mdim)})} \rceil$.
Since the distribution $\mprob_1$ is uniform over a support of at most
$2^{4\pdim}$ points, every point has probability mass no less than
$(1/2)^{4 \pdim}$.  This fact, combined with the lower
bound~\eqref{EqnBearCreekSpire}, implies that for any $\uvec \in
\Sset_1$ the solver $\ustar_\MYINT$ satisfies the bound $\norm{M
  \ustar_\MYINT -M \uvec}_2 < 1/2$ with probability greater than
$1/2$.

Note that the solver $\ustar_\MYINT$ can be encoded in $\MYINT \,
\pfuna(\mdim)$ bits and it terminates in $\MYINT \, \pfunb(\mdim)$ time,
both are polynomial functions of $\mdim$.  According to
Lemma~\ref{lemma:ppoly-consequence}
and~Lemma~\ref{lemma:np-hardness-of-solving-linear-equation}, there is
an integer $\mdim > \mdim_0$ such that on at least one input $\uvec$,
the solver $\ustar_\MYINT$ fails to achieve $\norm{M \ustar_\MYINT -M \uvec }_2
< 1/2$ with probability $1/2$.  However, this contradicts the previous
conclusion, which implies that our starting assumption was incorrect.
\end{proof}


\subsection{Proof of part (a)}
\label{sec:define-distribution}

We now turn to the construction of the design matrix $\Xmat$ specified
in Theorem~\ref{ThmMain}.  Lemma~\ref{lemma:probablistic-hardness}
implies that given arbitrary polynomial functions $\pfuna, \pfunb$ and
$\pfunc$, there is a positive integer $\mdim > \mdim_0$ (where
$\mdim_0$ can be arbitrarily large), a matrix $M \in \R^{(\mdim + 3
  \pdim )\times 4\pdim}$ with $\pdim = {\mdim \choose 3}$, and a
distribution $\qprob$ over the set of binary vectors \mbox{$\uvec \in
  \{0,1\}^{4\pdim} \cap \Ball_0( \frac{\mdim}{3} + \pdim)$} such that
solving for $\uvec$ based on observing $(M,M \uvec)$ is hard in the
average case. We leave $\pfuna, \pfunb$, $\pfunc$ and $\mdim_0$ to be
specified later, assuming that they are sufficiently large.  Based on
integer $\mdim$ and for a second positive integer $\tdim$ to be
chosen, we define the sparsity level
\begin{align*}
\kdim & \defn \tdim \, \Big ( \frac{\mdim}{3} + \pdim \Big) \: = \:
\tdim \, \Big (\frac{\mdim}{3} + {\mdim \choose 3} \Big).
\end{align*}
Using the matrix $M$ as a building block, we construct a design matrix
$\Xmat \in \R^{\numobs\times \usedim}$ that satisfies the conditions of
Theorem~\ref{ThmMain}.  To this end, our first step is to construct a
matrix $A_k$.  Recall the matrix $M \in \R^{(\mdim + 3 \pdim )\times
  4\pdim}$ constructed in Section~\ref{sec:construct-np-hard-problem}.
For the given integer $\tdim > 1$, we consider the rescaled matrix
$\sqrt{\tdim} M$ and replicate this matrix $\tdim$ times in diagonal
blocks to build the following matrix:
\begin{align*}
  \Akmat \defn \mbox{blkdiag} \underbrace{\big \{ \sqrt{\tdim} M,
    \sqrt{\tdim} M, \ldots, \sqrt{\tdim} M \big \}}_{\mbox{$\tdim$
      copies}} \in \real^{3 \kdim \times 4 \pdim \tdim}.
\end{align*}
By construction, the matrix $\Akmat$ has $3k = \tdim (\mdim + 3
\pdim)$ rows and $4 \pdim \tdim$ columns.  Since $4 \pdim \, \tdim
\leq 4 \kdim \leq \usedim$, the matrix
\begin{align*}
\Bkmat & \defn \frac{1}{2} \, \begin{bmatrix} \Akmat & 0_{3
    \kdim \times (\usedim - 4 \pdim \tdim)}
\end{bmatrix} \qquad \mbox{has dimensions $3 \kdim \times \usedim$.}
\end{align*}

We now use $\Bkmat$ to construct the design matrix $\Xmat \in
\real^{\numobs \times \usedim}$.  (To simplify the argument, we assume
that $\numobs$ is divisible by $6 \kdim$; when this condition does not
hold, an obvious modification of the argument yields the same result
modulo different constants.) Assume that $R\in \R^{(\numobs/2)\times
  \usedim}$ is a random Gaussian matrix whose rows are sampled
i.i.d.~from the Gaussian distribution $N(0,I_{\usedim \times
  \usedim})$. We define a parameterized family of random matrices $ \{
\Cmat_x\in \R^{\numobs\times \usedim}, x \geq 0 \}$ such that:
\begin{itemize}
\item its top $\numobs/2$ rows consist of $\frac{\numobs}{6 \kdim}$
  copies of the matrix $\Bkmat$.
\item for each $x \geq 0$, the bottom $\numobs/2$ rows of $\Cmat_x$
  are equal to $x R$.
\end{itemize}
Given a matrix $C \in \R^{\numobs\times\usedim}$, let $\Cmatup$ and
$\Cmatdown$ represent its top $\frac{\numobs}{2}$ rows and the bottom
$\frac{\numobs}{2}$ rows, respectively. In addition, we use
$\floor{C}_\ldim$ to represent the matrix formed by quantizing every
entry of $C$ by the operator $\floor{\cdot}_\ldim$. In addition, we
define $\supersmall \defeq 2^{-\pFunb(\numobs,\usedim,\kdim)}$ as a
shorthand quantity and note that $\RECON\in [\supersmall,
  \frac{1}{24\sqrt{2}})$.  The following lemma shows that there exists
  a particular value of $x$ such that the matrix $\Cmat_x$ has nice
  properties.

\begin{lemma}
\label{lemma:existence-of-nice-matrix}
For any integer $\ldim \geq
\max\{\log(12\sqrt{\usedim}),\log(\sqrt{\numobs\usedim}/\supersmall)\}$,
there is a realization of the random matrix $R$ and a particular value
of $x$ such that the matrix $\floor{\Cmat_x}_\ldim$ has RE constant
$|\RECON(\floor{\Cmat_x}_\ldim) - \RECON| \leq \supersmall$, and
satisfies the upper bounds
\begin{align}
\label{eqn:design-matrix-normalization-property}
  \frac{\|\floor{\Cmat_x}_\ldim \theta\|_2}{\sqrt{\numobs}} \leq
  \|\theta\|_2 \quad \mbox{and} \quad \frac{\|\floor{\Cxdown}_\ldim
    \theta\|_2}{\sqrt{\numobs}} \leq 25\RECON \|\theta\|_2 \quad
  \mbox{for all $\theta \in \Ball_0(2 \kdim)$}.
\end{align}
\end{lemma}

According to Lemma~\ref{lemma:existence-of-nice-matrix}, if we choose
an integer $\ldim \geq \max
\{\log(12\sqrt{\usedim}),\log(\sqrt{\numobs\usedim}/\supersmall)\}$
and define the design matrix $\Xmat \defeq \floor{\Cmat_x}_\ldim$,
then the matrix $\Xmat$ satisfies part (a) of the theorem. We leave
the concrete value of $\ldim$ to be specified later, assuming now that
it is sufficiently large.


\subsection{Proof of part (b)}
\label{SecProofB}

We begin with some notation for this section.  Let $\Xup$ and $\Xdown$
represent the top $\frac{\numobs}{2}$ rows and the bottom
$\frac{\numobs}{2}$ rows of $\Xmat$, respectively.  Given a vector
$\theta \in \R^\usedim$, we take its first $4 \pdim \tdim$ coordinates
and partition them evenly into $\tdim$ segments, each of length $4
\pdim$. For $i \in [\tdim]$, we use $\INDEXER{\theta}{i} \in \real^{4
  \pdim}$ to denote the $i^{th}$ segment, corresponding to the
coordinates indexed $4 \pdim (i-1)+1$ to $4\pdim i$.  Using this
notation, the proof involves the subvectors $\{\INDEXER{\thetahat}{i},
\; i \in [\tdim]\}$ and $\{\INDEXER{\theta}{i}^*, \; i \in [\tdim]\}$
defined by $\thetahat$ and $\thetastar$, respectively.

Lemma~\ref{lemma:probablistic-hardness} guarantees the existence of a
distribution $\qprob$ over $\ANNOY$ such that, if $\uvec^*$ is sampled
from $\qprob$, then solving $M u = M\uvec^*$ is computationally hard.
Now define the quantity
\begin{align}
\label{EqnLchoice}
\ldim \defeq \left \lceil \max\left\{\log(12
\sqrt{\usedim}),\log(\sqrt{\numobs \usedim} / \supersmall), \log(2
\sqrt{\kdim / \tdim}/\rdim)\right\} \right \rceil \quad \mbox{and}
\quad \rho \defn \left \lfloor \frac{\rdim}{\sqrt{\mdim/3 + \pdim}}
\right \rfloor_L,
\end{align}
where $\rdim > 0$ is a constant to be specified later. Suppose that
each subvector $\INDEXER{\theta}{i}^*/\rho$ is independently drawn
from the distribution $\qprob$---to be precise, a subvector is drawn
from the distribution $\qprob$, and is rescaled by a factor of $\rho$
to obtain $\theta_i^*$.  Note that
\begin{align*}
\max_{\thetastar \in \Ball_0(\kdim)} \E_{\thetastar} \Big[ \|\Xmat
  \thetahat(\Xmat, y) - \Xmat \thetastar\|_2^2 \Big] & \geq
\E_\qprob \Big[ \|\Xmat \thetahat(\Xmat, y) - X\thetastar\|_2^2 \Big],
\end{align*}
so that it suffices to lower bound the right-hand side.

Denote by $\pone$ the problem defined in
Lemma~\ref{lemma:probablistic-hardness}, and denote by $\ptwo$ an
instance of the linear regression problem with fixed design matrix
$X$.  Our reduction from $\pone$ to $\ptwo$ consists of two
main steps:

\begin{itemize}
\item We first define an intermediate problem \pprimetwo~that is close
  to but slightly different than our regression problem $\ptwo$.  As
  we demonstrate, any \pprimetwo~instance can be constructed from a
  $\pone$ instance, and hence by
  Lemma~\ref{lemma:probablistic-hardness}, it is hard to solve
  \pprimetwo~up to a particular accuracy.
\item Our next step is to show that (a) any estimator $\thetahat$
  solving the regression problem $\ptwo$~also solves the intermediate
  problem \pprimetwo; and (b) the prediction error associated with
  problems $\ptwo$~and \pprimetwo~are of the same magnitude.
  Consequently, there is a lower bound on the error of $\thetahat$ for
  solving the regression problem $\ptwo$.
\end{itemize}

In detail, there are two main differences between the linear equation
$M u = M \uvec^*$ from Lemma~\ref{lemma:probablistic-hardness} and the
linear observation model (see equation~\eqref{EqnPtwoPrime}) that
underlies \pprimetwo.  First, in addition to the basic matrix $M$, the
design matrix $\Xmat$ contains additional $n/2$ bottom rows consisting
of independent Gaussian entries, Second, in the linear regression
problem, the response vector $\yvec$ is corrupted by Gaussian noise,
whereas the original linear system $Mu = M\uvec^*$ is based on a
noiseless response vector (namely, $M\uvec^*$).  Our reduction from
$\pone$ to \pprimetwo~ bridges these gaps.


\paragraph{Constructing a \pprimetwo~instance from a \pone~instance:}

Suppose that we are given an instance $(M,M\uvec^*)$ of problem \pone,
where the vector $\uvec^*$ is sampled from $\qprob$.  The input of the
\pprimetwo~problem takes the form $(\Xmat, \yvecprime)$, where
$\Xmat\in\R^{\numobs\times\usedim}$ is a matrix depending on $M$, and
$\yvecprime\in \R^\numobs$ is a vector depending on
$(M,M\uvec^*)$. The goal of both \pone~and \pprimetwo~is to recover
the unknown vector $\uvec^*$.  We now construct a set of
\pprimetwo~instances $(\Xmat,y_i)$, indexed by
$i\in\{1,2,\dots,\tdim\}$, and we show that solving any of them
provides a solution to \pone.  We have already discussed how to
construct the design matrix $\Xmat$ in
Section~\ref{sec:define-distribution}, so that it remains to construct
the response vector $\yvecprime_i$. By focusing on a concrete index
$i$, we suppress the dependence on the index $i$ and write $\yvecprime
\defeq \yvecprime_i$ for the conciseness of writing. Recalling our
notation $\{\theta_j, j \in [\tdim]\}$ for the subvectors of the
vector $\theta \in \real^\usedim$, our construction consists of the
following three steps:
\begin{enumerate}
 \item[(i)] Take as input a constant vector $\TSTARI \in \R^\usedim$
   provided by an external source.

\item[(ii)] Form the vector $\IOPTI \in \real^{\usedim}$ with
  subvectors specified as
\begin{align*}
\IOPTI_\ell & \defn \begin{cases} \rho \uvec^* & \mbox{if $\ell = i$}
  \\ \TSTARI_\ell & \mbox{if $\ell \in [\tdim] \backslash i$,}
\end{cases}
\end{align*}
with all other coordinates set to zero.
\item[(iii)] Form the response vector
\begin{align}
\label{EqnPtwoPrime}
\yvecprime & = \begin{bmatrix} \Xup \: \IOPTI \\ \Xdown \: \TSTARI
    \end{bmatrix} + w,
\end{align}
where $w\sim N(0,\sigma^2I_{n\times n})$ is a Gaussian noise vector,
Note that this is not a standard linear regression model, since a
different regression vector is used for the top and bottom $n/2$
responses (the regression vectors $\IOPTI$ and $\TSTARI$,
respectively).
\end{enumerate}
It is important to note that even though $\IOPTI$ contains $\rho
\uvec^*$ as a subvector, constructing the matrix vector product $\Xup
\IOPTI$ only requires the knowledge of matrix vector product $M
\uvec^*$. The external constant vector $\TSTARI$ from step (i) can be
seen as an ``advice string'', as permitted for solving problems in the
$\ppoly$~class (recall the definition in Section~\ref{SecMain}). Note
that its value is independent of the $\pone$~instance $(M,M\uvec^*)$.

\paragraph{Solving \pprimetwo~via an estimator for \ptwo:}

Having constructed a \mbox{\pprimetwo~instance} $(\Xmat,\yvecprime)$
from $(M, M\uvec^*)$, we now show how an estimator for solving a
$\ptwo$ instance can be used to solve the \pprimetwo~instance, and
hence to solve the original \mbox{$\pone$ instance.}  More precisely,
given an estimator $\thetahat$ for problem \mbox{$\ptwo$,} we output
$\PLAINSOLVER_i \defn \thetahat_i(X,\yvecprime)/\rho$ as a solution to
\pprimetwo, estimating the unknown vector $\uvec^*$.  The following
lemma provides an upper bound on the probability of the ``success
event'' that the outputted solution $\PLAINSOLVER_i$ is
$(\kdim/\tdim)$-sparse, and achieves a small prediction error. Its
proof exploits the hardness result from
Lemma~\ref{lemma:probablistic-hardness}.

\begin{lemma}
\label{lemma:lower-bound-by-p1-hardness}
Assume that $\max\{\rdim, 1/\rdim\}$ is bounded by a polynomial
function of $(\numobs,\usedim,\kdim)$, and let $\pfunc$ be an
arbitrary polynomial function.  Then for any vector $\theta \in
\R^\usedim$ such that $\norms{\theta}_\infty \leq \rho$ and $\theta =
\lfloor \theta \rfloor_L$, we have
\begin{align}
\label{eqn:lower-bound-by-p1-hardness}
\mprob \Big[\PLAINSOLVER_i\in \Ball_0(\kdim/\tdim) \mbox{ and }
  \norm{M\PLAINSOLVER_i - M\uvec^*}_2 < 1/2 ~\mid \TSTARI = \theta
  \Big] \leq \frac{1}{h(m)}.
\end{align}
\end{lemma}

\noindent Note that upper bound~\eqref{eqn:lower-bound-by-p1-hardness}
holds for any assignment to $\TSTARI$.  Thus, it holds if we randomize
it by drawing $\TSTARI \in \R^\usedim$ randomly from a probability
distribution.  In particular, suppose that we construct arandom vector
$\TSTARI$ as follows:
\begin{itemize}
\item for each $j \in [\tdim] \backslash\{i\}$, we draw the $j^{th}$
  sub-vector $\TSTARI_j/\rho \in \real^{4 \pdim}$ independently from
  distribution~$\qprob$.
\item All other coordinates of $\TSTARI$ are set to zero.
\end{itemize}
It is straightforward to verify that the random vector $\TSTARI$
constructed in this way satisfies $\norms{\TSTARI}_\infty \leq \rho$
and $\TSTARI = \lfloor \TSTARI \rfloor_L$.  Thus, integrating over the
randomness of $\TSTARI$, Lemma~\ref{lemma:lower-bound-by-p1-hardness}
implies
\begin{align}
\label{eqn:lower-bound-by-p1-hardness-unconditioned}
  \mprob \Big[\PLAINSOLVER_i\in \Ball_0(\kdim/\tdim) \mbox{ and }
    \norm{M\PLAINSOLVER_i - M\uvec^*}_2 < 1/2\Big] \leq
  \frac{1}{h(m)}.
\end{align}
We assume for the rest of the proof that $\TSTARI$ is drawn from the
specified distribution. By this construction, it is important to note
that $\IOPTI$ and $\thetastar$ have the same distribution, because
each of their $4\pdim$-length blocks are independently drawn from the
same distribution~$\qprob$. To prove the Thoerem~\ref{ThmMain}, it
remains to convert the hardness
result~\eqref{eqn:lower-bound-by-p1-hardness-unconditioned} to a lower
bound on the mean-squared error $\E \Big[ \|\Xmat \thetahat(\Xmat, y)
  - X\thetastar\|_2^2 \Big]$.

\paragraph{Proving a lower bound on \ptwo~via the hardness of \pprimetwo:}
For the purposes of analysis, it is convenient to introduce the
auxiliary response vector $\ydoub \defn \Xmat \IOPTI + w$.  Since the
random vectors $\IOPTI$ and $\thetastar$ have the same distribution,
the response vectors $\yvec \defeq \Xmat \thetastar + w$ and $\ydoub$
have the same distribution.  We also define an auxiliary estimator
$\AUXSOLVER_i \defn \thetahat_i(\Xmat,\ydoub)/\rho$. Since $\yvec$ and
$\ydoub$ share the same probability distribution, we note that the
estimator $\AUXSOLVER_i$ shares the same distribution as
$\thetahat_i(\Xmat, \yvec)/\rho$. Consequently, we have
\begin{align}
 &\E \left[\frac{1}{\numobs} \|\Xmat \thetahat(\Xmat,y)- \Xmat
    \thetastar\|_2^2 \right] \geq \E \left[ \frac{1}{\numobs} \,
    \|\Xup \thetahat(\Xmat,y) - \Xup \thetastar||_2^2 \right] \;
  \nonumber\\ &\qquad \stackrel{(i)}{=} \; \frac{1}{\numobs} \,
  \frac{\numobs}{6 \kdim} \, \frac{\tdim}{2} \sum_{i=1}^\tdim \E \big[
    \|M \thetahat(\Xmat,y)_i - M \thetastar_i\|_2^2 \big]
  \stackrel{(ii)}{=} \frac{\tdim \rho^2}{12\kdim} \sum_{i=1}^\tdim \E
  \big[ \|M \AUXSOLVER_i - M \uvec^*\|_2^2
    \big] \label{eqn:from-local-matrix-to-random-matrix}.
\end{align}
Here, equality (i) holds because $\frac{\numobs}{6 \kdim}$ is the
number of replicates of the submatrix $B_\kdim$ in the full design
matrix $\Xmat$, and $\frac{\tdim}{2}$ is the scale factor in the
definition of $B_\kdim$. Equality (ii) holds because the pair
$(\AUXSOLVER_i, \uvec^*)$ has the same probability distribution as the
pair $(\thetahat_i(\Xmat, \yvec)/\rho, \thetastar_i/\rho)$. Hence, it
suffices to lower bound the right-hand side of
inequality~\eqref{eqn:from-local-matrix-to-random-matrix}.\\

In order to utilize the hardness result of
inequality~\eqref{eqn:lower-bound-by-p1-hardness-unconditioned}, we
need to show that $\|M \AUXSOLVER_i - M \uvec^*\|_2^2$ and $\|M
\PLAINSOLVER_i - M \uvec^*\|_2^2$ are of the same order of
magnitude. Note that the estimator $\AUXSOLVER_i$ takes
$(\Xmat,\ydoub)$ as input, whereas the estimator $\PLAINSOLVER_i$
takes $(\Xmat,\yvecprime)$ as input.
Since the top $\numobs/2$ coordinates of $\ydoub$ are identical to
those of $\yvecprime$, one might expect that the outputs of
$\AUXSOLVER_i$ and $\PLAINSOLVER_i$ are strongly correlated; it
remains to prove this intuition.

Using $\BOT{\ydoub}$ and $\BOT{\yvecprime}$ to denote the bottom
$\numobs/2$ co-ordinates of $\ydoub$ and $\yvecprime$ respectively, it
is easy to check that $\BOT{\ydoub} \sim N(\Xdown \IOPTI, \sigma^2
I_{\numobs/2 \times \numobs/2})$ and $\BOT{\yvecprime} \sim N(\Xdown
\TSTARI, \sigma^2 I_{\numobs/2 \times \numobs/2})$. The means of
$\BOT{\ydoub}$ and $\BOT{\yvecprime}$ differ by $\Xmat(\IOPTI -
\TSTARI)$.  Moreover, we know by construction that $\ltwos{\IOPTI -
  \TSTARI} = \ltwos{\rho \uvec^*} \leq \rdim$.  These observations are
sufficient for lower bounding the right-hand side of
inequality~\eqref{eqn:from-local-matrix-to-random-matrix}.\\

We now state two lemmas that characterize the properties of the
estimator $\AUXSOLVER_i$ and the
estimator $\PLAINSOLVER_i$. The first lemma shows that a large subset of the estimators
$\{ \PLAINSOLVER_i, i \in [\tdim]\}$ return sparse outputs with
positive probability. See Appendix~\ref{AppLemSparseEst} for the proof of this
claim.
\begin{lemma}
\label{lemma:sparse-estimator-with-high-probability}
There is a set $\Tset \subset \{1,\dots,t\}$ with cardinality $|\Tset|
\geq \frac{t^2}{2 \kdim+\tdim}$ such that
\begin{align*}
\mprob \Big[\PLAINSOLVER_i \in \Ball_0(\kdim/\tdim) \Big] & \geq
e^{-1} \Big \{ \frac{\tdim}{2 \kdim +2 \tdim} - \frac{25\RECON r
  \sqrt{\numobs}}{\sigma} \Big \} \qquad \mbox{for each $i \in
  \Tset$.}
\end{align*}
\end{lemma}

\vspace{10pt}
\noindent To understand the consequence of Lemma~\ref{lemma:sparse-estimator-with-high-probability}, we examine
the probability that $\norms{M\PLAINSOLVER_i - M\uvec^*}_2$ is lower bounded by $1/2$.
By the union bound, we have
\begin{align*}
  \mprob \Big[ \norm{M\PLAINSOLVER_i - M\uvec^*}_2 \geq 1/2\Big] + \mprob\Big[\PLAINSOLVER_i \notin \Ball_0(\kdim/\tdim) \Big] \geq 1 - \mprob \Big[\PLAINSOLVER_i\in \Ball_0(\kdim/\tdim) \mbox{ and } \norm{M\PLAINSOLVER_i - M\uvec^*}_2 < 1/2\Big],
\end{align*}
and consequently,
\begin{align*}
  \mprob \Big[ \norm{M\PLAINSOLVER_i - M\uvec^*}_2 \geq 1/2\Big] &\geq \mprob\Big[\PLAINSOLVER_i \in  \Ball_0(\kdim/\tdim) \Big] - \mprob \Big[\PLAINSOLVER_i\in \Ball_0(\kdim/\tdim) \mbox{ and } \norm{M\PLAINSOLVER_i - M\uvec^*}_2 < 1/2\Big].
\end{align*}
According to Lemma~\ref{lemma:sparse-estimator-with-high-probability}, there is a set $\Tset$
with cardinality $|\Tset| \geq \frac{t^2}{2 \kdim+\tdim}$, such that $\mprob\Big[\PLAINSOLVER_i \in  \Ball_0(\kdim/\tdim) \Big]$
is lower bounded by a positive constant for all $i\in \Tset$.
Combining the lower bound with the above inequality and inequality~\eqref{eqn:lower-bound-by-p1-hardness-unconditioned},
we obtain
\begin{align}\label{eqn:lower-bound-independent-of-sparsity}
  \mprob \Big[ \norm{M\PLAINSOLVER_i - M\uvec^*}_2 \geq 1/2\Big] &\geq e^{-1} \Big \{ \frac{\tdim}{2 \kdim +2 \tdim} - \frac{25\RECON r
  \sqrt{\numobs}}{\sigma} \Big \} - \frac{1}{\pfunc(m)} \qquad \mbox{for each $i \in
  \Tset$.}
\end{align}
Inequality~\eqref{eqn:lower-bound-independent-of-sparsity} provides a lower bounded on the norm $\norm{M\PLAINSOLVER_i - M\uvec^*}_2$,
independent of the sparsity level of the estimator~$\PLAINSOLVER_i$.

The second lemma upper bounds the prediction error of $\PLAINSOLVER_i$ in terms of the
mean-squared prediction error of the auxiliary solver $\AUXSOLVER_i$. See Appendix~\ref{AppLemErrorControl} for the proof of this
claim.
\begin{lemma}
\label{lemma:error-control}
As long as $\frac{\RECON r \sqrt{n}}{\sigma} \leq 1/50$, we are
guaranteed that
\begin{align*}
\norm{M\PLAINSOLVER_i - M\uvec^*}_2^2 \leq \frac{2e\sigma}{25\RECON r
  \sqrt{\numobs}} \, \E\Big[\norm{M\AUXSOLVER_i - M\uvec^*}_2^2\Big].
\end{align*}
with probability at least $1 - \frac{50\RECON r\sqrt{n}}{\sigma}$.
\end{lemma}
\vspace{10pt}

By combining Lemma~\ref{lemma:error-control} and inequality~\eqref{eqn:lower-bound-independent-of-sparsity},
we obtain a lower bound on the mean-squared prediction error of the auxiliary solver $\AUXSOLVER_i$.
In particular, for each index $i\in \Tset$, we have
\begin{align}\label{eqn:sub-problem-mse-lower-bound-before-assignment}
  \E\Big[\norm{M\AUXSOLVER_i - M\uvec^*}_2^2\Big] \geq \frac{1}{4}\cdot \frac{25\RECON r
  \sqrt{\numobs}}{2e\sigma}
\end{align}
with probability at least
\begin{align*}
   e^{-1}\Big\{ \frac{\tdim}{2 \kdim +2 \tdim} - \frac{25\RECON r
  \sqrt{\numobs}}{\sigma} \Big \} - \frac{50\RECON r\sqrt{n}}{\sigma} - \frac{1}{\pfunc(m)}.
\end{align*}
We proceed by assigning particular values to the parameters
$\tdim$ and $\rdim$.  In particular, we set
\begin{align}
\label{EqnRchoice}
&\rdim \, \defn \frac{\sigma}{25\times 16\RECON\sqrt{n}}\:\frac{t}{k+t}, \quad
\mbox{and} \quad \tdim \defn \lceil \big(\frac{\mdim}{3} + \pdim
\big)^{\frac{1-\alpha}{\alpha}} \rceil,
\end{align}
where $\alpha \in (0,1)$ is to be chosen later. We note that the assignment to $\rdim$ satisfies
the assumption of Lemma~\ref{lemma:lower-bound-by-p1-hardness}.
Recalling that $\kdim
= \tdim \, (\frac{\mdim}{3}+ \pdim)$ by definition, we have the
sandwich relation
\begin{align}
\label{EqnSandwich}
\kdim^{1-\alpha} \; \leq \; \tdim \;
\leq \; \kdim,
\end{align}
as well as the inequalities
\begin{subequations}
\begin{align}
 e^{-1} \Big( \frac{\tdim}{2 \kdim +2 \tdim} - \frac{25\RECON r
   \sqrt{\numobs}}{\sigma}\Big) - \frac{50\RECON r
   \sqrt{\numobs}}{\sigma} & = \Big( \frac{7}{16e} - \frac{1}{8}\Big)
 \frac{\tdim}{\kdim + \tdim} > 0.03 \, \kdim^{-\alpha}, \qquad
 \mbox{and} \label{eqn:marginal-probability-lower-bound}\\
 \frac{25\RECON r
  \sqrt{\numobs}}{2e\sigma} = \frac{\tdim}{32 e \kdim + 32 e
   \tdim} & \geq \frac{\kdim^{-\alpha}}{64 e }.
\end{align}
\end{subequations}
By inequality~\eqref{eqn:marginal-probability-lower-bound}, we verify that
the condition of Lemma~\ref{lemma:error-control} is satisfied.
Combining with lower bound~\eqref{eqn:sub-problem-mse-lower-bound-before-assignment}, and by choosing a
polynomial function $\pfunc$
such that
$1/\pfunc(m) < 0.03 \kdim^{-\alpha}$,
we see that there is a subset
$\Tset$ with cardinality \mbox{$|\Tset|\geq \frac{1}{3} \kdim^{1-2\alpha}$} such that for each $i \in \Tset$,
\begin{align*}
  \mprob \Big[ \E\Big[\norm{M\AUXSOLVER_i - M\uvec^*}_2^2\Big] \geq
\frac{\kdim^{-\alpha}}{256 e } \Big] > 0.
\end{align*}
Since $\E[\norm{M\AUXSOLVER_i - M\uvec^*}_2^2]$ is a numerical value, the above
inequality holds if and only if it is greater than or equal to
$\frac{\kdim^{-\alpha}}{256 e}$. Plugging this lower bound into
inequality~\eqref{eqn:from-local-matrix-to-random-matrix}, and substituting in
\begin{align*}
  \rho = \floorl{ \frac{\rdim}{\sqrt{\mdim/3+\pdim}} }_\ldim  \geq  \frac{\rdim}{\sqrt{\kdim/\tdim}} - 2^{-\ldim} \geq  \frac{\rdim}{2\sqrt{\kdim/\tdim}}
\end{align*}
as well as our previous
choice~\eqref{EqnRchoice} of the parameter $\rdim$, we find that
\begin{align}
\label{EqnHana}
  \E \left[\frac{1}{\numobs} \|\Xmat \thetahat(\Xmat, y)- \Xmat \thetastar\|_2^2
   \right] &\geq \frac{\rho^2 \tdim}{12 \kdim} \sum_{i \in \Tset} \E[\norm{M\AUXSOLVER_i - M\uvec^*}_2^2] \geq \frac{\rho^2 \tdim}{12 \kdim}\; \frac{\kdim^{1-2\alpha}}{3}\; \frac{\kdim^{-\alpha}}{256 e} \nonumber\\
   &
    \geq \frac{\rdim^2 \tdim^2 \kdim^{-1-3\alpha}}{4\times 3\times 256 e} \; \geq
    \frac{c}{\RECON^2}\, \frac{\sigma^2
      \kdim^{1-7\alpha}}{\numobs},
\end{align}
where we have also used the sandwich inequalities~\eqref{EqnSandwich}.
By assumption, the
problem dimension is upper bounded as $\usedim \leq \pFuna(\kdim)$
where $\pFuna$ is a polynomial function.  Since
Lemma~\ref{lemma:probablistic-hardness} allows $\mdim_0$ to be
arbrtrary, we may choose it large enough so as to ensure that
$\kdim^{\alpha} \geq \log \usedim$, or equivalently $\kdim^{1-7\alpha}
\geq \kdim^{1-8\alpha} \log \usedim$.  Finally, setting $\alpha =
\hackpar/8$ completes the proof of the theorem.


\section{Conclusion}
\label{sec:conclusion}

In this paper, under a standard conjecture in complexity theory, we
have established a fundamental gap between the prediction error
achievable by optimal algorithms, and that achievable by
polynomial-time algorithms.  In particular, whereas the prediction
error of an optimal algorithm has no dependence on the restricted
eigenvalue constant, our theory shows that the prediction error of any
polynomial-time algorithm exhibits an inverse dependence on this
quantity (for a suitably constructed design matrix).  To the best of
our knowledge, this is the first lower bound on the
polynomially-constrained minimax rate of statistical estimation that
depends only conjectures in worst-case complexity theory.


\subsection*{Acknowledgements}

This work was supported in part by NSF CISE Expeditions award
CCF-1139158, the U.S.\ Army Research Laboratory, the U.S.\ Army
Research Office under grant number W911NF-11-1-0391, the Office of
Naval Research MURI grant N00014-11-1-0688, and DARPA XData Award
FA8750-12-2-0331, and by gifts from Amazon Web Services, Google, SAP,
Apple, Inc., Cisco, Clearstory Data, Cloudera, Ericsson, Facebook,
GameOnTalis, General Electric, Hortonworks, Huawei, Intel, Microsoft,
NetApp, Oracle, Samsung, Splunk, VMware, WANdisco and Yahoo!.


\appendix


\section{Technical lemmas for Theorem~\ref{ThmMain}}
\label{SecTechLemma}

In this appendix, we collect together the proofs of various technical
lemmas involved in the proof of our main theorem.


\subsection{Proof of Lemma~\ref{lemma:ppoly-consequence}}
\label{AppPpolyConsequence}

Consider a problem $\PROB$ for which, for any integer $\mdim >
\mdim_0$, there is a Turing machine $\MACH_\mdim$ such that:
\begin{enumerate}
\item[(a)] the code length of $\MACH_\mdim$ is at most $\pfuna(\mdim)$
  bits, and;
\item[(b)] it solves every input of length $\mdim$ with probability at
  least $1/2$, and terminates in time at most $\pfunb(\mdim)$.
\end{enumerate}
Under Assumption~\ref{assumption:np-notin-ppoly}, it suffices to show
that these two conditions certify that $\PROB \in \ppoly$.

In order to prove this claim, let $\mdim > \mdim_0$ be arbitrary.
Given a sufficiently long random sequence, we may execute
$\MACH_\mdim$ a total of $\mdim+1$ times, with the randomness in each
execution being independent. By property (b), for any binary input of
length $\mdim$, each execution has success probability at least $1/2$.
By independence, the probability that at least one of the $\mdim +
1$ executions is successful is $1-2^{-\mdim-1}$.  Since the total
number of $\mdim$-length inputs is at most $2^\mdim$, the union bound
implies that the $\mdim+1$ executions of $\MACH_\mdim$ succeed on all
inputs with probability at least $1-2^{-\mdim-1}2^\mdim=1/2 >
0$. Consequently, the probabilistic method implies that there exists some
realization of the random sequence---call it $\RAND_\mdim$---under
which $\mdim+1$ executions of $\MACH_\mdim$ solves the problem for all
inputs of length $\mdim$.

Let $\CODE_\mdim$ be the code defining machine $\MACH_\mdim$, and
consider a Turing machine $\MACHPRIME_\mdim$ that takes the pair
$(\CODE_\mdim, \RAND_\mdim)$ as an advice string, and then uses the
string $\RAND_\mdim$ to simulate the execution of $\MACH_\mdim$ a
total of $\mdim + 1$ times.  From our previous argument, the Turing
machine $\MACHPRIME_\mdim$ solves the problem on every input of length
$\mdim >\mdim_0$.  Notice that the length of string $(\CODE_\mdim,
\RAND_\mdim)$ is of the order $\order(\pfuna(\mdim) + \mdim
\pfunb(\mdim))$ bits, and the running time of $\MACHPRIME_\mdim$ is of
the order $\order(\mdim \pfunb(\mdim))$. Finally, for all input
lengths $\mdim \leq \mdim_0$, a constant-size Turing machine
$\MACH_\mdim^\star$ can solve the problem in constant time.  The
combination of $\MACHPRIME_\mdim$ and $\MACH_\mdim^\star$ provides a
Turing machine that certifies $\PROB \in \ppoly$.


\subsection{Proof of Lemma~\ref{lemma:existence-of-nice-matrix}}

Let $\tau = 8 \sqrt{2} \RECON$, then let $\Ctauup$ and $\Ctaudown$
represent the top $\frac{\numobs}{2}$ rows and the bottom
$\frac{\numobs}{2}$ rows of $\Ctaumat$, respectively. We first prove
an auxiliary lemma about the random matrix $\Cmat_\tau$.

\begin{lemma}\label{lemma:nice-random-matrix}
There are universal positive constants $\UNICON_1,\UNICON_2$ such
that, with probability at least \mbox{$1-\UNICON_1\exp(-\UNICON_2\numobs)$,}
the random matrix $\Cmat_\tau$ has the following properties:
\begin{itemize}
\item It satisfies the RE
condition~\eqref{EqnDefnRE} with parameter~$\RECON$.
\item It satisfies the following upper bounds:
\begin{align}
\label{EqnSuperAnnoyStrong}
  \frac{\|\Ctauup \theta\|_2^2}{\numobs} \leq
  \frac{1}{3}\|\theta\|_2^2 \quad \mbox{and} \quad \frac{\|\Ctaudown
    \theta\|_2^2}{\numobs} \leq \left(24\RECON\right)^2\|\theta\|_2^2
  \leq \frac{1}{2}\|\theta\|_2^2 \quad \mbox{for all $\theta \in
    \Ball_0(2 \kdim)$}
\end{align}
\end{itemize}
\end{lemma}

\begin{proof}
Dealing first with the upper block of the matrix $C$, we have
\begin{align}
\label{eqn:upper-bound-X-by-A}
\norm{ \Ctauup\theta}_2^2 & = \frac{1}{4} \, \frac{\numobs}{2} \,
\frac{1}{3 \kdim} \, \norm{\Akmat \theta}_2^2.
\end{align}
Recall that the diagonal blocks of $\Akmat$ are matrices $M$.  The
following claim provides an upper bound on its singular values:
\begin{claim}
\label{lemma:restricted-eigenvalue}
For any vector $\uvec \in \R^{4 \pdim}$, we have $\| M \uvec \|_2^2
\leq 8 \pdim \|\uvec\|_2^2$.
\end{claim}
\begin{proof}
Since the entries of $M$ all belong to $\{-1,0,1\}$, if we let $M_{i}$
denote the $i^{th}$ row of $M$, then
\begin{align*}
  (M_{i}\uvec)^2 \leq \big( \sum_{j \, \mid \, |M_{ij}|=1} |\uvec_j|
  \big)^2 \leq \big| \{j \, \mid \, |M_{ij}|=1 \} \big| \;
  \|\uvec\|_2^2.
\end{align*}
Summing over all indices $i =1,2,\dots, \mdim +3 \pdim$, we find that
\begin{align*}
  \| M\uvec\|^2 \leq \big| \{(i,j) \, \mid |M_{ij}|=1 \} \big| \;
  \|\uvec\|_2^2.
\end{align*}
Since there are $8 \pdim$ nonzero entries in $M$, the claim
follows.
\end{proof}

Returning to the main thread, Claim~\ref{lemma:restricted-eigenvalue}
implies that $\Akmat$ has singular values bounded by $\sqrt{8 \tdim
  \pdim}$.  Putting together the pieces, we have
\begin{align}
\label{eqn:upper-bound-A-by-M}
  \| \Akmat \theta \|_2^2 & \leq \; 8 \, \tdim \, \pdim \: \|\theta\|_2^2 \;
  \leq \; 8 \, \kdim \|\theta \|_2^2,
\end{align}
where the final inequality follows since $\tdim \pdim \leq \tdim
(\frac{\mdim}{3} + \pdim) = \kdim$. In conjunction,
inequalities~\eqref{eqn:upper-bound-X-by-A}
and~\eqref{eqn:upper-bound-A-by-M} imply that $\norm{
  \Ctauup\theta}_2^2 \leq \numobs/3 \, \|\theta \|_2^2$.

Our next step is to prove that $\|\Ctaudown \theta\|_2^2 \leq
\numobs/2 \, \|\theta\|_2^2$ with high probability.  Since $\numobs
\succsim \kdim \log \usedim$, standard results for Gaussian random
matrices (see Lemma~\ref{lemma:eigenvalue-lower-upper-bound} in
Appendix~\ref{AppGaussRandMat}) imply that
\begin{align}
\label{eqn:eigenvalue-upper-bound-for-gaussian-part}
\norm{ \Ctaudown\theta}_2^2 & \leq \tau^2 \, (3 \sqrt{\numobs/2}
\norm{\theta}_2)^2 \; \stackrel{(i)}{\leq} \; \frac{\numobs}{2}
\|\theta\|_2^2 \qquad \mbox{for all $\theta \in \Ball_0(2 \kdim)$}
\end{align}
with probability at least $1 - \UNICON_1\exp(-\UNICON_2\numobs)$, where inequality (i) follows since $\tau = 8
\sqrt{2} \RECON \leq 1/3$. Combining the upper bounds for $\Ctauup$ and
the upper bound for $\Ctaudown$, we find that the normalization
conditions~\eqref{EqnSuperAnnoyStrong} hold with high probability.

It remains to show that the RE condition~\eqref{EqnDefnRE} holds with
high probability.  By Lemma~\ref{lemma:eigenvalue-lower-upper-bound},
we have
\begin{align*}
\| \Ctaumat \theta\|_2^2 & \geq \| \Ctaudown \theta \|_2^2 \; \geq \; \tau^2
\big(\frac{\sqrt{\numobs/2}}{8} \|\theta\|_2\big)^2 \; = \; \RECON^2
\, \numobs \, \|\theta\|_2^2,
\end{align*}
a bound that holds uniformly for all $\theta \in \bigcup_{\substack{
    \PlainSset \subset\{1,\dots, \usedim \} \\ |\PlainSset|= \kdim}}
\ConeSet(\PlainSset)$, as required by the condition~\eqref{EqnDefnRE}.
\end{proof}

According to Lemma~\ref{lemma:nice-random-matrix}, there is a
realization of the random matrix $R$ such that the matrix $\Cmat_\tau$
satisfies both the condition~\eqref{EqnSuperAnnoyStrong}, and the RE
condition~\eqref{EqnDefnRE} with parameter~$\RECON$. We take this
realization of $R$ to define a concrete realization of $\Cmat_\tau$.
Consequently, the RE constant of this matrix, namely
$\RECON(\Cmat_\tau)$, satisfies $\RECON(\Cmat_\tau) \geq \RECON$.
For arbitrary matrices $X,Y\in \R^{\numobs\times \usedim}$, it is straightforward that
\begin{align}
\label{eqn:RE-constant-continuity}
  |\RECON({X}) - \RECON(Y)| \leq \opnorm{X - Y},
\end{align}
showing that the function $X \mapsto \RECON(X)$ is a Lipschitz
function with parameter $1$, and hence continuous. Consequently,
$\RECON(\Cmat_x)$ is a continous function of $x$.
 We also claim that
it satisfies the condition
\begin{align}
\label{eqn:zero-recon-claim}
  \RECON(\Cmat_0) = 0,
\end{align}
a claim proved at the end of this section.  Based on the continuity
property~\eqref{eqn:RE-constant-continuity} and the initial
condition~\eqref{eqn:zero-recon-claim}, there is a constant $\tau'\in
(0, \tau]$ such that $\RECON(\Cmat_{\tau'})= \RECON$.  Since
  $\Cmat_\tau$ satisfies the normalization
  condition~\eqref{EqnSuperAnnoyStrong} and $\tau'\leq \tau$, we have
\begin{align}
\label{eqn:normalization-transfer}
  \frac{\|\Cmat_{\tau'} \theta\|_2^2}{\numobs} \leq \frac{\|\Ctaumat
    \theta\|_2^2}{\numobs} \leq \frac{5}{6}\|\theta\|_2^2 \quad
  \mbox{and} \quad \frac{\|\Cmat_{\tau'}^\downarrow
    \theta\|_2^2}{\numobs} \leq \frac{\|\Ctaudown
    \theta\|_2^2}{\numobs} \leq (24\RECON)^2\|\theta\|_2^2 \quad
  \mbox{for all $\theta \in \Ball_0(2 \kdim)$}.
\end{align}

Next we consider the quantized matrix
$\floor{\Cmat_{\tau'}}_\ldim$. Since the quantization operator
approximates every entry of $\Cmat_{\tau'}$ to precision $2^{-\ldim}$,
we have
\begin{align}
\label{eqn:matrix-quantization-error}
  \opnorm{\floor{\Cmat_{\tau'}}_\ldim - \Cmat_{\tau'}} \leq
  \frobnorm{\floor{\Cmat_{\tau'}}_\ldim - \Cmat_{\tau'}} \leq
  2^{-\ldim}\sqrt{\numobs\usedim}.
\end{align}
Combining this inequality with the Lipschitz
condition~\eqref{eqn:RE-constant-continuity}, we find that
\begin{align*}
|\RECON(\floor{\Cmat_{\tau'}}_\ldim) - \RECON| \leq
2^{-\ldim}\sqrt{\numobs\usedim},
\end{align*}
showing that the quantized version satisfies the stated RE condition
as long as $\ldim \geq \log(\sqrt{\numobs\usedim}/\supersmall)$.
Turning to the normalization conditions,
inequality~\eqref{eqn:normalization-transfer} and
inequality~\eqref{eqn:matrix-quantization-error} imply that
\begin{align*}
  \frac{\|\floor{\Cmat_{\tau'}}_\ldim \theta\|_2}{\sqrt{\numobs}} \leq
  \frac{\|\Cmat_{\tau'}\theta\|_2 +
    2^{-\ldim}\sqrt{\numobs\usedim}\norm{\theta}_2}{\sqrt{\numobs}}
  \leq (\sqrt{5/6} + 2^{-\ldim}\sqrt{\usedim}) \norm{\theta}_2 \quad
  \mbox{and}\\ \frac{\|\floor{\Cmat_{\tau'}^\downarrow}_\ldim
    \theta\|_2}{\sqrt{\numobs}} \leq
  \frac{\|\Cmat_{\tau'}^\downarrow\theta\|_2 +
    2^{-\ldim}\sqrt{\numobs\usedim}\norm{\theta}_2}{\sqrt{\numobs}}
  \leq (24\RECON + 2^{-\ldim}\sqrt{\usedim}) \norm{\theta}_2 \quad
  \mbox{for all $\theta \in \Ball_0(2 \kdim)$}.
\end{align*}
It is straightforward to verify that if $\ldim \geq
\max\{\log(12\sqrt{\usedim}),\log(\sqrt{\numobs\usedim}/\supersmall)\}$,
then the normalization
conditions~\eqref{eqn:design-matrix-normalization-property} stated in
the lemma are also satisfied.

\paragraph{Proof of equation~\protect\eqref{eqn:zero-recon-claim}:}

Notice that $\rank(\Cmat_0) = \rank(\Akmat) \leq 3\kdim$, but the
number of nonzero columns of $\Cmat_0$ is greater than $3\kdim$.
Hence, there is a $(3\kdim+1)$-sparse nonzero vector $\thetasparse$
such that
\begin{align}
\label{eqn:zero-recon-evidence-1}
    \Cmat_0 \thetasparse = 0.
\end{align}
Let $\PlainSset$ be a set of indices that correspond to the $\kdim$
largest entries of $\thetasparse$ in absolute value.  Then
$|\PlainSset| = \kdim$ and $\norm{\thetasparse_{\PlainSset^c}}_1 \leq
\frac{2k+1}{k}\norm{\thetasparse_{\PlainSset}}_1$. The later
inequality implies that
\begin{align}
\label{eqn:zero-recon-evidence-2}
    |\PlainSset| = \kdim ~~~~\mbox{and}~~~~ \thetasparse\in
    \ConeSet(\PlainSset).
\end{align}
Combining expressions~\eqref{eqn:zero-recon-evidence-1}
and~\eqref{eqn:zero-recon-evidence-2} yields that $\RECON(\Cmat_0)=0$.


\subsection{Proof of Lemma~\ref{lemma:lower-bound-by-p1-hardness}}
\label{AppLemP1Hard}

\newcommand{\trunlevel}{\mu}

Recall that the noise vector $w$ follows the normal distribution
$N(0,\sigma^2I_{n\times n})$.  Consequently, if we define the event
$\event_1 \defn \{ \|w\|_\infty \leq \trunlevel \}$, where
\begin{align*}
  \trunlevel \defeq \max\left\{ \frac{2\numobs}{\sqrt{2 \pi}}, \sqrt{2\log(2 h(m))} \right\},
\end{align*}
then standard tail
bounds and the union bound imply that
\begin{align*}
  \mprob[\event_1^c] \leq \numobs \, \big \{ \frac{2}{\sqrt{2 \pi}
      \trunlevel} e^{-\trunlevel^2/2} \big \} \leq \frac{1}{2 h(m)}.
\end{align*}

Let $\event_2$ denote the event that $\PLAINSOLVER_i$ being $(\kdim/\tdim)$-sparse
and that $\norm{M\PLAINSOLVER_i - M\uvec^*}_2 < 1/2$.
We split the probability $\mprob[\event_2]$ into two terms. We find that
\begin{align*}
  \mprob[\event_2] &= \int_{\|z\|_\infty \leq \trunlevel}\mprob[\event_1 | w = z]\phi(z; 0,\sigma^2I_{n\times n}) {\rm d}z + \mprob[\event_2 | \event_1^c] \mprob[\event_1^c]\\
  &= \int_{\|z\|_\infty \leq \trunlevel}\mprob[\event_1 | w = z]\phi(z; 0,\sigma^2I_{n\times n}) {\rm d}z +  \frac{1}{2 h(m)}
\end{align*}
where  $\phi(z; 0,\sigma^2I_{n\times n})$ denotes the density of the $N(0,\sigma^2I_{n\times n})$ distribution.
Consequently, there must be some specific value $\norm{w_0}_\infty \leq \trunlevel$ such that
\begin{align}\label{eqn:split-success-event}
 \mprob[\event_2] \leq \mprob[\event_2 | w = w_0] +  \frac{1}{2 h(m)} .
\end{align}

Let $\solver^*_i$ be the solver that uses the deterministic argument
$w_0$ rather than drawing $w$ randomly according to the specified
procedure.  The remaining steps followed by $\solver^*_i$ are exactly
the same as those of $\PLAINSOLVER_i$.  We now demonstrate that the
solver $\solver^*_i$ has polynomial complexity in terms of code length
and running time:

\paragraph{Code length:} We begin by encoding the
design matrix $\Xmat$ and fixed vectors $(\TSTARI,w_0)$ into the
program of the solver $\solver^*_i$. Recall that both $\Xmat$ and
$\TSTARI$ are discretized to $2^{-\ldim}$ precision. For any discrete
value $x$ of $2^{-\ldim}$ precision, encoding it takes at most
$\log(|x|+1) + \ldim$ bits. Recalling our choice~\eqref{EqnLchoice} of
$\ldim$, is straightforward to verify that $\ldim = \order(\log({\rm
  poly}(\numobs,\usedim,\kdim))+\log(1/\supersmall))$.  Moreover, the
components of the matrix $\Xmat$ and vector $\TSTARI$ are adequately
controlled: in particular, by
inequality~\eqref{eqn:design-matrix-normalization-property},
equation~\eqref{EqnLchoice} and the definition of $\TSTARI$, each such entry $x$ has
an encoding bounded by
$\log(|x|+1)= \order({\rm
  poly}(\numobs,\usedim,\kdim))+\log(1/\supersmall))$, Since
$\log(1/\supersmall) = \pFunb(\numobs,\usedim,\kdim)$ is a polynomial
functions of $m$, we find that both $\Xmat$ and $\TSTARI$ have
polynomial code length.

Now the noise vector $w_0$  is involved in computing the
response
vector
\begin{align*}
  \yvecprime & = \begin{bmatrix} \Xup \: \IOPTI \\ \Xdown \: \TSTARI
    \end{bmatrix} + w_0.
\end{align*}
Since the estimator $\thetahat$ takes
$\floor{\yvecprime}_{\pFunb(\numobs,\usedim,\kdim)}$ as input, it
suffices to quantize $w_0$ sufficiently finely so as to ensure that
the quantized response
$\floor{\yvecprime}_{\pFunb(\numobs,\usedim,\kdim)}$ is not altered by
the quantization.  Note that the components of $\Xmat$, $\TSTARI$ and
$\IOPTI$ all have $2^{-\ldim}$ precision.  Consequently, their products
have at most $2^{-2\ldim}$ precision, and we have
\begin{align*}
  \left \lfloor \begin{bmatrix} \Xup \: \IOPTI \\ \Xdown \: \TSTARI
    \end{bmatrix} + w_0 \right\rfloor_{\pFunb} =
\left \lfloor \begin{bmatrix} \Xup \: \IOPTI \\ \Xdown \: \TSTARI
    \end{bmatrix} + \floor{w_0}_{\max\{2\ldim, \pFunb\}} \right\rfloor_{\pFunb},
\end{align*}
showing that it suffices to quantize the noise $w_0$ at level
$\floor{\cdot}_{\max\{2\ldim, \pFunb\}}$. On the other hand, by construction,
the absolute values of $w_0$ are uniformly bounded by $\trunlevel$. Hence, the
vector $w_0$ can also be encoded in polynomial code length.

Finally, we have assumed that code length of $\thetahat$ is bounded
by $\Lint$, and observed that the program that constructs $\yvecprime$ is of
constant code length.  Consequently, the total code length is bounded
by a polynomial function of $\mdim$, which we denote by
$\pfuna(\mdim)$.

\paragraph{Running time:}
Since the response vector $\yvecprime$ can be constructed from $(M,
M\uvec^*, \Xmat, \TSTARI, w_0)$ in polynomial time and the solver
$\thetahat$ terminates in polynomial time (bounded by
$\pFunc(\size(\Xmat,\yvecprime; \pFunb))$), we conclude that
$\solver^*_i$ also terminates in polynomial time, say upper bounded by
the polynomial $\pfunb(\mdim)$.\\


\noindent To complete the argument, we utilize the hardness result of Lemma~\ref{lemma:probablistic-hardness}.
Indeed, this lemma guarantees that a polynomial-complexity solver cannot achieve probability of success greater than $1/\pfunc(\mdim)$,
which yields
\begin{align*}
 \mprob[\event_2 | w = w_0] \leq  \frac{1}{2 h(m)}.
\end{align*}
Combining this upper bound with inequality~\eqref{eqn:split-success-event} completes
the proof.


\subsection{Proof of Lemma~\ref{lemma:sparse-estimator-with-high-probability}}
\label{AppLemSparseEst}

Let $\qdensprime$
and $\qdensdoub$ denote the probability density functions of
$\yvecprime$ and $\ydoub$. With this notation, we have
\begin{align*}
\mprob \big[ \PLAINSOLVER_i \in \Ball_0(\kdim/\tdim) \big] & =
\int_{\R^\numobs} \Ind \big[\thetahat(\Xmat, z) \in
  \Ball_0(\kdim/\tdim) \big] \: \qdensprime(z) dz \\
& = \int_{\R^\numobs} \frac{\qdensprime(z)}{\qdensdoub(z)} \Ind
\big[\thetahat(\Xmat, z) \in \Ball_0(\kdim/\tdim) \big] \:
\qdensdoub(z) dz \\
& = \E \Biggr[ \frac{\qdensprime(\ydoub)}{\qdensdoub(\ydoub)} \; \Ind
  \big[\AUXSOLVER_i \in \Ball_0(\kdim/\tdim) \big] \Biggr] \\
& \geq e^{-1} \, \mprob \Big[\AUXSOLVER_i \in \Ball_0(\kdim/\tdim)
  \mbox{ and } \frac{\qdensprime(\ydoub)}{\qdensdoub(\ydoub)} \geq
  e^{-1} \Big] \\
& \geq e^{-1} \; \Big \{ \mprob \big[\AUXSOLVER_i \in
  \Ball_0(\kdim/\tdim) \big] - \mprob \Big[
  \frac{\qdensprime(\ydoub)}{ \qdensprime(\ydoub)} < e^{-1} \Big] \Big
\}.
\end{align*}
We claim that
\begin{align}\label{eqn:claim-rare-event-bound}
 \mprob \Big[ \frac{\qdensprime(\ydoub)}{\qdensdoub(\ydoub)} < e^{-1}
   \Big] & \leq \frac{25\RECON \rdim \sqrt{\numobs}}{\sigma}.
\end{align}
Deferring the proof of upper bound~\eqref{eqn:claim-rare-event-bound}, we now focus
on its consequence. If $\mprob \big[\AUXSOLVER_i \in \Ball_0(\kdim/\tdim) \big] \geq
\frac{\tdim}{2 \kdim + 2 \tdim}$, then we are guaranteed that
\begin{align}
\label{eqn:prob-sparse-lower-bound}
  \mprob \big[\PLAINSOLVER_i \in \Ball_0(\kdim/\tdim) \big] & \geq
  e^{-1} \, \Big( \frac{\tdim}{2 \kdim + 2 \tdim} - \frac{25\RECON \rdim
    \sqrt{\numobs}}{\sigma}\Big).
\end{align}

Finally, let $\Tset \subseteq \{1, \ldots, \tdim \}$
be the subset of indices for which $\AUXSOLVER_i \in
\Ball_0(\kdim/\tdim)$ with probability at least $\frac{\tdim}{2 \kdim
  + 2 \tdim}$.  Using $\thetahat_i$ as a shorthand for
$\thetahat_i(\Xmat,y)$, note that $\AUXSOLVER_i$ and
$\thetahat_i/\rho$ share the same distribution, and hence
\begin{align}
\label{EqnEquiv}
\mprob \big[\AUXSOLVER_i \in \Ball_0(\kdim/\tdim) \big] =
\mprob\big[\thetahat_i \in \Ball_0(\kdim/\tdim) \big].
\end{align}
Since $\thetahat \in \Ball_0(\kdim)$ by construction, we have
\begin{align}
\kdim & \geq \|\thetahat\|_0 \; \geq \; \sum_{i=1}^\tdim
\|\thetahat_i\|_0 \geq \big( \frac{\kdim}{\tdim} + 1 \big) \,
\sum_{i=1}^\tdim \Ind \big[\thetahat_i \notin \Ball_0(\kdim/\tdim)
  \big] \; = \; \frac{\kdim + \tdim}{\tdim} \Big \{ \tdim - \sum_{i =
  1}^\tdim \Ind \big[\thetahat_i \in \Ball_0(\kdim/\tdim) \big] \Big
\}.
\end{align}
Following some algebra, we find that $\sum_{i=1}^\tdim \, \Ind
\big[\thetahat_i \in \Ball_0(\kdim/\tdim) \big] \geq
\frac{\tdim^2}{\kdim + \tdim}$, and hence
\begin{align*}
\sum_{i=1}^\tdim \mprob[\thetahat_i \in \Ball_0(\kdim/\tdim)] & \geq
\frac{\tdim^2}{\kdim+\tdim}.
\end{align*}
Let $N$ be the number of indices such that $\mprob[\thetahat_i \in
  \Ball_0(\kdim/\tdim)] \geq \frac{\tdim}{2 \kdim+2 \tdim}$.  By
definition, there are $\tdim - N$ indices for which
$\mprob[\thetahat_i \in \Ball_0(\kdim/\tdim)] < \frac{\tdim}{2 \kdim +
  2\tdim}$, and hence
\begin{align*}
\frac{\tdim^2}{\kdim + \tdim} \leq \sum_{i=1}^\tdim \mprob
\big[\thetahat_i \in \Ball_0(\kdim/\tdim) \big] \leq N + (\tdim - N)\:
\frac{\tdim}{2 \kdim+2 \tdim},
\end{align*}
which implies that $N \geq \frac{\tdim^2}{2 \kdim+ \tdim}$, as
claimed.  By the equivalence~\eqref{EqnEquiv}, we have $N = |\Tset|$,
so that the proof is complete.

\paragraph{Proof of inequality~\eqref{eqn:claim-rare-event-bound}} Recall that $\yvecprime$ and $\ydoub$ differ in distribution only in
their last $\numobs/2$ entries, denoted by $\BOT{\yvecprime}$ and
$\BOT{\ydoub}$ respectively.  Since $\BOT{\yvecprime} \sim N( \Xdown
\TSTARI, \sigma^2 I)$ and $\BOT{\ydoub} \sim N(\Xdown \IOPTI,
\sigma^2I)$, we have
\begin{align}
\label{EqnHanaSleep}
\frac{\qdensprime(\ydoub)}{\qdensdoub(\ydoub)} & = \exp
\Big( - \frac{(\Xdown \TSTARI - \BOT{\ydoub})^2}{2 \sigma^2} +
\frac{(\Xdown \IOPTI - \BOT{\ydoub})^2}{2\sigma^2}\Big).
\end{align}
Define the scalar $a \defn \norm{\Xdown(\IOPTI - \TSTARI)}_2/\sigma$,
and define the event $\Event \defn \{ Z \leq \frac{a}{2} - \frac{1}{a} \}$
where $Z \sim N(0,1)$ is a standard normal variate.  From the
representation~\eqref{EqnHanaSleep}, some algebra shows that the probability that the upper
bound $\frac{\qdensdoub(\yvecprime)}{\qdensprime(\yvecprime)} \leq
e^{-1}$ holds is equal to the probability that the event $\Event$ holds.  Letting $\Phi$
denote the CDF of a standard normal variate, we have $\mprob[\event] =
\Phi \big(\frac{a}{2} - \frac{1}{a}\big)$.  It can be verified
that $\Phi \big(\frac{a}{2} - \frac{1}{a} \big) \leq a$ for all $a
\geq 0$, whence
\begin{align*}
  \mprob[\event] \leq a & = \frac{\|\Xdown( \IOPTI -
    \TSTARI)\|_2}{\sigma} \; \leq \; \frac{25\RECON \sqrt{\numobs}
    \, \| \IOPTI - \TSTARI\|_2}{\sigma},
\end{align*}
where the last step uses
inequality~\eqref{eqn:design-matrix-normalization-property}. Since
$\IOPTI$ and $\TSTARI$ differ only in one subvector---more precisely,
by $\rho \uvecstar$---we have
\begin{align*}
  \mprob[\event] & \leq \frac{25 \RECON \sqrt{\numobs} \|\rho
    \uvec^*\|_2}{\sigma} \leq \frac{25\RECON \rdim
    \sqrt{\numobs}}{\sigma},
\end{align*}
where the last step follows since $\norm{\rho \uvec^*}_2 \leq \rdim$.


\subsection{Proof of Lemma~\ref{lemma:error-control}}
\label{AppLemErrorControl}

Letting $\qdensprime$ and $\qdensdoub$ denote the probability density
functions of $\yvecprime$ and $\ydoub$, respectively, we have
\begin{align}
  \E\big[\| M \AUXSOLVER_i - M \uvec^*\|_2^2\big] & =
  \int_{\R^\numobs} \norms{M \thetahat_i(\Xmat,z) - M \uvec^*}_2^2 \:
  \qdensdoub(z) dz \nonumber \\
& = \int_{\R^\numobs}
  \frac{\qdensdoub(z)}{\qdensprime(z)}\norms{M\thetahat_i(X,z) -
    M\uvec^*}_2^2 \: \qdensprime(z) dz \nonumber \\
& = \E \Big[\frac{\qdensdoub(\yvecprime)}{\qdensprime(\yvecprime)} \:
    \|M \PLAINSOLVER_i - M \uvec^* \|^2\Big] \nonumber \\
& \geq e^{-1} \: \E \Big[ \|M \PLAINSOLVER_i - M \uvec^*\|_2^2 \, \mid
    \, \frac{\qdensdoub(\yvecprime)}{\qdensprime(\yvecprime)}\geq
    e^{-1}\Big] \; \mprob
  \Big[\frac{\qdensdoub(\yvecprime)}{\qdensprime(\yvecprime)} \geq
    e^{-1} \Big].
\label{eqn:control-error-using-a-proxy-distribution}
\end{align}
Following precisely the proof of inequality~\eqref{eqn:claim-rare-event-bound}, we can prove
the upper bound
\begin{align*}
 \mprob \Big[ \frac{\qdensdoub(\yvecprime)}{\qdensprime(\yvecprime)} < e^{-1}
   \Big] & \leq \frac{25\RECON \rdim \sqrt{\numobs}}{\sigma}.
\end{align*}
As a consequence, we have
\begin{align*}
  \mprob \Big[ \frac{\qdensdoub(\yvecprime)}{\qdensprime(\yvecprime)} \geq e^{-1} \Big] \geq 1 - \frac{25\RECON \rdim \sqrt{\numobs}}{\sigma} \geq \frac{1}{2},
\end{align*}
where the last inequality follows since $\frac{\RECON \rdim \sqrt{\numobs}}{\sigma} \leq \frac{1}{50}$ by assumption.
Plugging this
lower bound into
inequality~\eqref{eqn:control-error-using-a-proxy-distribution} yields
\begin{align*}
\E \Big[ \|M \PLAINSOLVER_i - M\uvec^* \|_2^2 \mid \frac{\qdensdoub(\yvecprime)}{\qdensprime(\yvecprime)} \geq e^{-1} \Big] &
\leq 2 e \, \E \big[\|M \AUXSOLVER_i - M \uvec^*\|_2^2 \big].
\end{align*}
Applying Markov's inequality yields
\begin{align*}
\mprob \Big[\|M \PLAINSOLVER_i - M \uvec^*\|_2^2 \leq \frac{
    \sigma}{25 \RECON \rdim \sqrt{\numobs}} \cdot 2 e\: \E \big[ \|M \AUXSOLVER_i
    - M\uvec^*\|_2^2 \big] \mid \frac{\qdensdoub(\yvecprime)}{\qdensprime(\yvecprime)} \geq e^{-1} \Big] & \geq 1 - \frac{25\RECON
  \rdim \sqrt{\numobs}}{\sigma}.
\end{align*}
Consequently, with probability $(1 - \frac{25\RECON
  r\sqrt{n}}{\sigma})^2 \geq 1 - \frac{50\RECON r\sqrt{n}}{\sigma}$, we
have
\begin{align*}
\|M \PLAINSOLVER_i - M \uvec^*\|_2^2 \leq \frac{2e \sigma}{25\RECON
  \rdim \sqrt{\numobs}} \: \E \big[\| M \AUXSOLVER_i - M \uvec^*\|_2^2
  \big],
\end{align*}
as claimed.


\section{Proof of Proposition~\ref{PropLassoThresh}}
\label{AppLassoThresh}

We begin by defining the error vectors $\PertHat \defn \thetaregparn -
\thetastar$ of the ordinary Lasso, as well as that $\PertTil \defn
\thetathr - \thetastar$ of the thresholded Lasso.  By construction,
the error vector $\PertTil$ is at most $2 \kdim$-sparse, so that the
normalization condition~\eqref{EqnSuperAnnoy} guarantees that
\begin{align}
\label{EqnSuperAnnoyBound}
\frac{\|\Xmat \PertTil\|_2^2}{\numobs} & \leq \|\PertTil\|_2^2.
\end{align}
The following lemma shows that the truncated Lasso is essentially as
good as the non-truncated Lasso:
\begin{lemma}
\label{LemAux}
The error of the truncated Lasso is bounded as $\|\PertTil\|_2^2 \leq
5 \|\PertHat\|_2^2$.
\end{lemma}
\noindent
We return to prove this intermediate lemma at the end of this
section. \\

Taking Lemma~\ref{LemAux} as given for the moment, we complete the
proof by bounding the error of the ordinary Lasso using Corollary~2
in~\cite{NegRavWaiYu12}.  With the specified choice of $\regparn$,
this corollary implies that
\begin{align}
\label{eqn:l2-upper-bound-without-sparsity-constraint}
\|\PertHat\|_2^2 & \leq \frac{64}{\RECONSQ} \; \frac{\sigma^2 \kdim
  \log \usedim}{\numobs},
\end{align}
with probability at least $1 - 2 e^{-\UNICON_4 \kdim \log \usedim}$.
Combining the bound~\eqref{eqn:l2-upper-bound-without-sparsity-constraint},
Lemma~\ref{LemAux} and inequality~\eqref{EqnSuperAnnoyBound} yields
the claim of the proposition.

Finally, we return to prove Lemma~\ref{LemAux}.  Throughout this
proof, we use $\thetahat$ as a shorthand for the thresholded estimator
$\thetathr$.  Since $\thetahat$ is the $\kdim$-truncated version of
$\thetaregparn$, if \mbox{$|\supp(\thetaregparn)| \leq \kdim$,} we
must have $\thetahat = \thetaregparn$, in which case the proof is
complete.  Otherwise, we may assume that $|\supp(\thetaregparn)| >
\kdim$, and then define the set $\HACKSET \defn \big
\{\supp(\thetastar) \cap \supp(\thetaregparn) \big \} \backslash
\supp(\thetahat)$, corresponding to ``good'' entries in the estimate
$\thetaregparn$ that were lost in the truncation.  Introduce the
notation $\HACKSET = \{i_1, \ldots, i_\tdim\}$ where $\tdim =
|\HACKSET|$.  Since $|\supp(\thetastar)| \leq \kdim$, the set
$\supp(\thetaregparn) \backslash \supp(\thetastar)$ must contain at
least $\tdim$ indices, say $B = \{j_1, \ldots, j_\tdim\}$,
corresponding to ``bad'' entries of $\thetaregparn$ that were
preserved in the truncation.  By definition of the truncation
operation, they satisfy the bound
\begin{align}
\label{eqn:truncated-index-value-comparison}
  |(\thetaregparn)_{i_s}| \leq |(\thetaregparn)_{j_s}| \qquad
  \mbox{for all $s = 1, \ldots, t$.}
\end{align}
Now consider the decomposition
\begin{align}
\label{EqnDecomposition}
\|\thetahat - \thetastar\|_2^2 & = \sum_{i=1}^\usedim (\thetahat_i -
\thetastar_i)^2 = \sum_{s=1}^t
\Big((\thetahat_{i_s}-\thetastar_{i_s})^2 +
(\thetahat_{j_s}-\thetastar_{j_s})^2\Big) + \sum_{i \notin B \cup
  \HACKSET} (\thetahat_i - \thetastar_i)^2.
\end{align}
The proof will be complete if we can establish the two inequalities
\begin{subequations}
\begin{align}
\label{EqnInqA}
(\thetahat_i - \thetastar_i)^2 & \leq ((\thetaregparn)_i -
\thetastar_i)^2 \quad \mbox{if $i \notin B \cup \HACKSET$, and} \\
\label{EqnInqB}
  (\thetahat_{i_s} - \thetastar_{i_s})^2 + (\thetahat_{j_s} -
\thetastar_{j_s})^2 & \leq 5 ( (\thetaregparn)_{i_s} -
\thetastar_{i_s})^2 + 5((\thetaregparn)_{j_s} - \thetastar_{j_s})^2
\qquad \mbox{for $s = 1, \ldots, t$.}
\end{align}
\end{subequations}
Inequality~\eqref{EqnInqA} is straightforward, since $i \notin B
\cup \HACKSET$ implies either $\thetahat_i = (\thetaregparn)_i$ or
$\thetahat_i = \thetastar_i = 0$.

Turning to the inequality~\eqref{EqnInqB}, note that
$\thetahat_{i_s}=\thetastar_{j_s} = 0$, so that it is equivalent to
upper bound $(\thetastar_{i_s})^2 + (\thetaregparn)^2_{j_s}$.  We
divide the proof of inequality~\eqref{EqnInqB} into two cases:

\paragraph{Case 1:}  First, suppose that
 $|(\thetaregparn)_{i_s}- \thetastar_{i_s}| \leq
\frac{1}{2}|\thetastar_{i_s}|$.  In this case, we have
$|(\thetaregparn)_{i_s}| \geq \frac{1}{2}|\thetastar_{i_s}|$, and
hence by inequality~\eqref{eqn:truncated-index-value-comparison}, we
obtain $|(\thetaregparn)_{j_s}| \geq \frac{1}{2}|\thetastar_{i_s}|$.
Combining the pieces yields
\begin{align*}
(\thetastar_{i_s})^2 + ((\thetaregparn)_{j_s})^2 & \leq 5
  ((\thetaregparn)_{j_s})^2 \; \leq \; 5 ( ( \thetaregparn)_{i_s} -
  \thetastar_{i_s})^2 + 5 ( (\thetaregparn)_{j_s}-\thetastar_{j_s})^2.
\end{align*}

\paragraph{Case 2:}
On the other hand, if $|(\thetaregparn)_{i_s}- \thetastar_{i_s}| >
\frac{1}{2}|\thetastar_{i_s}|$, then we have
\begin{align*}
(\thetahat_{i_s} - \thetastar_{i_s})^2 + (\thetahat_{j_s} -
  \thetastar_{j_s})^2 & = (\thetastar_{i_s})^2 +
  ((\thetaregparn)_{j_s})^2 \\
& < 4((\thetaregparn)_{i_s}- \thetastar_{i_s})^2 +
  ((\thetaregparn)_{j_s})^2 \\
& \leq 5 ( (\thetaregparn)_{i_s} - \thetastar_{i_s})^2 + 5 (
  (\thetaregparn)_{j_s} - \thetastar_{j_s} )^2.
  \end{align*}
Combining the two cases completes the proof of
inequality~\eqref{EqnInqB}.


\section{Singular values of random matrices}
\label{AppGaussRandMat}

In this section, we provide some background on the singular value of
Gaussian random matrices, required in the proof of
Theorem~\ref{ThmMain}(a).  Our results apply to a random matrix $A \in
\real^{\numobs \times \usedim}$ formed of i.i.d. $N(0,1)$ entries.
Recall the set $\ConeSet(\PlainSset) = \big\{ \theta \in \real^\usedim
\, \mid \, \|\theta_{\Sbar}\|_1 \leq 3 \|\theta_{\PlainSset}\|_1 \big
\}$ that is involved in the RE condition~\eqref{EqnDefnRE}.

\begin{lemma}
\label{lemma:eigenvalue-lower-upper-bound}
Suppose that $\numobs > \UNICON_0 \, \kdim \log \usedim$ for a
sufficiently large constant $\UNICON_0$.  Then there are universal
constants $\UNICON_j, j = 1, 2$ such that
\begin{subequations}
\begin{align}
\label{eqn:eigenvalue-upper-bound}
\frac{\|A \theta\|_2}{\sqrt{\numobs}} & \leq 3 \|\theta\|_2 \qquad
\mbox{for all $\theta \in \Ball_0(2 \kdim)$, and } \\
\label{eqn:eigenvalue-lower-bound}
\frac{\|A\theta\|_2}{\sqrt{\numobs}} & \geq \frac{\|\theta\|_2}{8}
\qquad \mbox{for all $\theta \in \bigcup \limits_{\substack{
      \PlainSset \subset\{1,\dots, \usedim \} \\ |\PlainSset| =
      \kdim}} \ConeSet(\PlainSset)$},
\end{align}
\end{subequations}
where both bounds hold with probability at least $1- \UNICON_1 \exp(-
\UNICON_2 \numobs)$.
\end{lemma}

\begin{proof}
To prove the upper bound~\eqref{eqn:eigenvalue-upper-bound}, it suffices to show that $\max_{|\PlainSset| = 2 \kdim}
\frac{\opnorm{A_\Tset}}{\sqrt{\numobs}}\leq 3$, where $A_\Tset$ denotes the columns
of $A$ whose indices are in set $\Tset$, and $\opnorm{\cdot}$
denotes the maximum singular value of a matrix.  By known bounds on
singular values of Gaussian random matrices~\cite{rudelson2010non}, for any
subset $\Tset$ of cardinality $2 \kdim$, we have
\begin{align*}
\mprob \big [ \opnorm{A_\Tset} \geq \sqrt{\numobs} + \sqrt{2\kdim} + \delta \big] & \leq 2\exp(- \delta^2/2) \qquad \mbox{for all
  $\delta > 0$.}
\end{align*}
If we let $\delta = \sqrt{\numobs}$ and use the assumption that $2\kdim \leq \numobs$, the above bound
becomes
\begin{align*}
  \mprob \big [ \opnorm{A_\Tset} \geq 3\sqrt{\numobs} \big] & \leq 2\exp(-\numobs/2).
\end{align*}
Noting that there are ${\usedim \choose 2 \kdim} \leq e^{2 \kdim \log(
  \frac{e \usedim}{2 \kdim})}$ subsets of cardinality $2 \kdim$, the
claim thus follows by union bound.

On the other hand, the lower bound~\eqref{eqn:eigenvalue-lower-bound}
is implied by the main result of Raskutti et
al.~\cite{raskutti2010restricted}.

\end{proof}


\bibliographystyle{abbrvnat} \bibliography{bib}

\end{document}